\documentclass[final,5p,times,twocolumn]{elsarticle}
\biboptions{compress}
\journal{Automatica}

\usepackage{bm}
\usepackage{algorithmic}
\usepackage{algorithm}
\usepackage{array}
\usepackage{textcomp}
\usepackage{stfloats}
\usepackage{url}
\usepackage{verbatim}
\usepackage{graphicx}
\usepackage{subfigure}
\usepackage{lipsum}
\usepackage{tabularx}
\usepackage{multirow}

\usepackage{amsmath,amssymb,amsfonts,amsthm} 
\usepackage{xcolor}
\usepackage{color}

\usepackage[bookmarks=true, colorlinks, citecolor=cyan, linkcolor=cyan]{hyperref}

\allowdisplaybreaks[1]

\newtheorem{theorem}{Theorem}
\newtheorem{corollary}{Corollary}
\newdefinition{lemma}{Lemma}
\newdefinition{assumption}{Assumption}
\newdefinition{remark}{Remark}
\newdefinition{definition}{Definition}
\newtheorem{proposition}{Proposition}

\begin{document}
\begin{frontmatter}
	
	\title{Distributed Riemannian Stochastic Gradient Tracking Algorithm on the Stiefel Manifold\tnoteref{footnoteinfo}} 
	
	\tnotetext[footnoteinfo]{This paper was not presented at any IFAC
		meeting. Zhao and Lei are with the Department of Control Science and Engineering, Tongji University, Shanghai, 201804, China; Lei is  also with the Shanghai Research Institute for Intelligent Autonomous Systems, Tongji University, Shanghai, 201210, China. Wang is with School of Electrical Engineering and Telecommunications, University of New South Wales, Sydney, NSW 2052, Australia.}
	
		\author{Jishu Zhao}\ead{2111115@tongji.edu.cn}  
		\author{Xi Wang}\ead{xi.wang14@unsw.edu.au}
		\author{Jinlong Lei}\ead{ leijinlong@tongji.edu.cn}

	\begin{abstract}                          
			This paper focus on investigating the distributed Riemannian stochastic optimization problem on the Stiefel manifold for multi-agent systems, where all the agents work collaboratively to optimize a function modeled by the average of their expectation-valued local costs. Each agent only processes its own local cost function and communicate with neighboring agents to achieve optimal results while ensuring consensus. Since the local Riemannian gradient in stochastic regimes cannot be directly calculated, we will estimate the gradient by the average of a variable number of sampled gradient, which however brings about noise to the system. We then propose a distributed Riemannian stochastic optimization algorithm on the Stiefel manifold  by combining the  variable sample size gradient approximation method with the gradient tracking dynamic. It is worth noticing that the suitably chosen  increasing sample size plays an important role in improving the algorithm efficiency, as it reduces the noise variance. In an expectation-valued sense, the iterates of all agents are proved to converge to a stationary point (or neighborhood) with fixed step sizes.  We further establish the convergence rate of the iterates for the cases when the sample size is exponentially increasing, polynomial increasing, or a constant, respectively. Finally, numerical experiments are implemented to demonstrate the theoretical results.
	\end{abstract}
\end{frontmatter}
	
	\section{Introduction}\label{sec-1}
	Distributed optimization has attracted much research attention in the past decades, motivated by  the need to solve optimization problems over large-scale datasets or complex multi-agent systems.
	It is worth noticing that some of the problems have particular manifold constraints, e.g. decentralized spectral analysis \cite{huang2020communication, kempe2004decentralized}, dictionary learning \cite{raja2015cloud}, and deep neural networks with orthogonal constraint \cite{vorontsov2017orthogonality}. Therefore, solving distributed optimization problems on Riemannian manifold has also received significant research attention over the past few years \cite{pmlr-v139-chen21g, deng2023decentralized, wang2022decentralized, doi:10.1137/20M1321000, chen2024decentralized, wang2022variance}.
	
	Let $\text{St}(n,r)$ denote the Stiefel manifold, and consider the function $f: \text{St}(n,r) \mapsto \mathbb{R}$. We focus on the following Riemannian stochastic optimization problem over connected networks:
	\begin{equation}\label{p-1}
	\begin{aligned}
	\min\ &f(X) := \frac{1}{N}\sum_{i=1}^N f_i(X),\\
	s.t.\ & X \in \mathcal{M} \triangleq \text{St}(n,r):=\{X \in \mathbb{R}^{n \times r}:X^\top X = I_r\},
	\end{aligned}
	\end{equation}
	where $I_r$ is the $r\times r$ identity matrix and the local cost function privately known by agent $i$ is an expectation-valued function defined as $f_i(X):=\mathbb{E}_{\xi_i} F_i(X,\xi_i)$. Though each agent merely has its local information, it can interact with its neighbors through a connected graph $\mathcal{G}=(\mathcal{N},\mathcal{E})$ to collaboratively solve the problem \eqref{p-1}.
	Here, $\mathcal{N}:=\{1, \dots , N\}$ denotes the index set including all agents,  and $\mathcal{E} \subset \mathcal{N} \times \mathcal{N}$ denotes the set of all communication links. If $(i,j)\in \mathcal{E}$, then node $i$ can send  information to node $j$.
	Problem \eqref{p-1} extends the distributed stochastic optimization problem in Euclidean space to the Stiefel manifold, which however is non-convex in Euclidean space. Next, we initiate a review of prior work.
	
	\subsection{Related Works}
	Stochastic distributed optimization has particular research interests among various distributed optimization frameworks, where the local cost functions are the expectation of functions with stochastic variables. Since the expectation-valued function has no closed-form, it could be prohibitive or costly to compute the exact gradient in big-data scenarios. Stochastic approximation (SA), originated from Robbins and Monro in the 1950s \cite{robbins1951stochastic}, is a commonly used method in algorithm design. The main idea of SA is to utilize a stochastic gradient to approximate the exact gradient. As such, stochastic gradient descent (SGD) has received extensive attention since it is practical and performs well in large-scale learning \cite{bottou2018optimization}. In recent years, distributed stochastic optimization in the Euclidean space has been widely studied where all the agents find the optimal solution cooperatively, such as \cite{nemirovski2009robust, srivastava2011distributed, pu2021distributed, sayed2014adaptation, lei2022distributed, alghunaim2019distributed}.
	
	However, it is unable to solve problem \eqref{p-1} directly using the aforementioned studies since the Stiefel manifold constraint lacks of convexity and linearity. The analysis of distributed optimization algorithms in the Euclidean space usually relies on the linear convergence of consensus protocol, which however is difficult to obtain on the manifolds. Since the Stiefel manifold can be viewed as an embedded submanifold in Euclidean space, there exist some recent literatures that solve this problem from the viewpoint of treating it in Euclidean space and develop new tools with the help of Riemannian optimization \cite{AbsilMahonySepulchre+2008, edelman1998geometry}. For example, \cite{pmlr-v139-chen21g} proposed two methods called the decentralized Riemannian SGD algorithm and the decentralized Riemannian gradient tracking algorithm, and established that their convergence rate are $\mathcal{O}(1/\sqrt{k})$ and $\mathcal{O}(1/k)$, respectively. In addition, \cite{wang2022decentralized} proposed a computation-efficient gradient tracking algorithm by using the augmented Lagrangian method, in which the iterations stay in a neighborhood of the Stiefel manifold but finally converge to the manifold with rate $\mathcal{O}(1/k)$. Furthermore, \cite{deng2023decentralized} used the projection operators instead of retractions and expanded the distributed Riemannian gradient descent algorithm and the gradient tracking version to the compact submanifolds of Euclidean space. Besides,  \cite{chen2024decentralized} proposed the first distributed Riemannian conjugate gradient algorithm and proved its global convergence over the Stiefel manifold.
	
	As far as we know, there are only two related works \cite{pmlr-v139-chen21g, wang2022variance} investigating the distributed and stochastic optimization settings on the Stiefel manifold. In \cite{pmlr-v139-chen21g}, the decentralized Riemannian SGD algorithm was proved to be convergent asymptotically with diminishing step sizes. In addition, \cite{wang2022variance} proposed an algorithm that converges with a $\mathcal{O}(1/k)$ rate, however, the iterates do not always stay on the Stiefel manifold and can only be restricted in a neighborhood of the manifold. Moreover, both \cite{pmlr-v139-chen21g} and \cite{wang2022variance} focused on the finite sum problem, while we are devoted to the on-line problem with expectation-valued cost functions. To the best of our knowledge, distributed algorithms for Riemannian stochastic optimization on the Stiefel manifold with an exact convergence under constant step sizes have not been achieved yet.
	
	\subsection{Contributions}
	
	This paper devoted to design an efficient algorithm for solving problem \eqref{p-1} under a connected network. We combine the dynamic gradient tracking method with a variable sample-size method and make the following contributions:
		
		(1) We propose a distributed Riemannian stochastic algorithm on the Stiefel manifold to solve problem \eqref{p-1}. Each agent estimates its local gradient by variable-sampled stochastic gradients, then obtains the search direction based on  the combination of multi-step consensus protocol and a gradient tracking dynamic.
		
		(2) Assuming  that the sampled gradients are unbiased with bounded variance and uniformly bounded on the Stiefel manifold. Suppose, in addition, that the local functions are L-smooth on Euclidean space and the initial points of the agents belong to the local region defined in Section \ref{sec-4}. The iterates of all agents are proved to remain in an invariant subset on the Stiefel manifold and eventually converge to a stationary point(or neighborhood) in expectation with fixed step sizes.
		
		(3) We further show that the variable number of samples has an impact on the convergence rate, since increasing sample size can reduce the noise variance. When an exponentially increasing sample size is utilized, the convergence rate is proved to be $\mathcal{O}(1/k)$, which is comparable to the deterministic framework. While for a constant sample size, the iterates converge to a neighborhood of the stationary point with the rate $\mathcal{O}(1/k)$. We further prove the convergence rate with polynomially increasing sample size as well. Besides, we establish the iteration and oracle complexity, as well as the number of communications, to achieve an $\epsilon$-stationary solution.

	The paper is organized as follows. Section \ref{sec-2} introduces some preliminaries of Riemannian geometry, the average on Stiefel manifold, and the optimality condition.  We reformulate the distributed Riemannian optimization problem \eqref{p-1} and design the algorithm in Section \ref{sec-3}, while the main results are given in Section \ref{sec-4}. Numerical experiments are provided in Section \ref{sec-5}, while some concluding remarks are given in Section \ref{sec-6}.
	
	\textbf{Notations.} Let $X^\top$ denote the transposition of $X$ and $tr(\cdot)$ denote the trace operator. Define the $N$-fold Cartesian product of $\mathcal{M}$ as $\mathcal{M}^N:=\mathcal{M}\times \cdots \times \mathcal{M}$. Let $I_n$ denote the $n\times n$ identity matrix and $\bold{1}_N$ denote the N-dimensional vector with all ones, respectively. In addition, we summarize the symbols used in this work in the following table (Table 1) for ease of reference.

\renewcommand{\arraystretch}{1.1}
\vspace{-0.5cm}
	\begin{table}[h]
		\centering
		\caption{Symbol notations}
		\begin{tabularx}{\columnwidth}{X|X}
			\hline
		Symbol & Notation  \\ \hline
		$\mathcal{M}$& Riemannian manifold\\ \hline
			$T_X \mathcal{M}, N_X \mathcal{M}$ & tangent space and its orthogonal complement at $X$ \\ \hline
		     $\nabla \phi(X)$  & Euclidean gradient \\ \hline
		     $\text{grad}\ \phi(X)$ & Riemannian gradient \\ \hline
		     $D\phi(X)[v]$ ($ v \in T_X \mathcal{M}$)& the directional derivative along $v$ \\ \hline
		     $\|\cdot\|_F$& Frobenius norm\\ \hline
		     $R_X(\cdot)$ & retraction map\\ \hline
		     $\mathcal{R}_X(\cdot)$& polar-retraction\\ \hline
		     $\mathcal{P}_{T_X \mathcal{M}}$ & orthogonal projection onto the tangent space\\ \hline
		     $\bar X$and $ \bold{\bar X}: = (\bm{1}_N \otimes I_n) \bar X$& the Euclidean mean and its stacking matrix\\ \hline
		     $\hat X$and $ \bold{\hat X}: = (\bm{1}_N \otimes I_n) \hat X$& induced arithmetic mean and its stacking matrix\\ \hline
		\end{tabularx}
	\end{table}

	\section{Preliminaries}\label{sec-2}
	\subsection{Riemannian geometry}\label{sec-2.1}
	Before presenting the algorithm, we first introduce some basic concepts and geometric properties of the Riemannian manifold in this section. The reader can refer to literature \cite{AbsilMahonySepulchre+2008} for more information.
	
	A topological space is called a manifold if each of the point has a neighborhood that is diffeomorphism to $\mathbb{R}^n$. If a smooth manifold $\mathcal{M}$ is equipped with a metric $\mathfrak{g}$, then $(\mathcal{M}, \mathfrak{g})$ is called a Riemannian manifold.  We focus on the Stiefel manifold in this context, which is $\text{St}(n,r):=\{X \in \mathbb{R}^{n \times r}:X^\top X = I_r\}$. $\text{St}(n,r)$ is compact and can be seen as the embedded submanifold on the Euclidean space.
	
	We denote the tangent space at a point $X \in \text{St}(n,r)$ by $T_X\mathcal{M}$, and its orthogonal complement with respect to the Euclidean space is denoted by the normal space $N_X\mathcal{M}$. Specifically, the tangent space of Stiefel manifold can be derived as $T_X\mathcal{M}:=\{u \in \mathbb{R}^{n \times r}: X^\top u + u^\top X = 0\}$\cite{edelman1998geometry}, which shows the tangent vectors are matrices. The metric on the tangent space $T_X\mathcal{M}$ is induced from the Euclidean inner product, which is $\mathfrak{g}(u,v) :=\langle u,v \rangle = \text{tr}( u^\top v )$, where $u,v \in T_X\mathcal{M}$. Then the induced norm is equivalent to the Frobenius norm, i.e., $\|u\|_g = \sqrt{\text{tr}(u^\top u) } = \|u\|_F$.
	
	We then introduce the concept of exponential map to connect the tangent space and manifold. An exponential map $\text{Exp}_X$ maps a tangent vector $u$ to a point $Y: =  \text{Exp}_X(u) \in \text{St}(n,r)$ such that there exist a geodesic between $X, Y$ in the direction of $u \in T_X \mathcal{M}$, where geodesic is a curve that locally minimizes the length. However, the computation of an exponential map is costly and has no closed form. The retraction $R_X$ is a first-order approximation of the exponential map with the requirement that 1) $R_X(0_X) = X$, where $0_X$ is the zero vector on $T_X\mathcal{M}$; 2) the differential of $R_X$ at $0_X$ is the identity map, i.e., $D R_X (0_X) = id_{T_X \mathcal{M}} $. For the Stiefel manifold, we choose the polar retraction, which is defined as \[\mathcal{R}_X(v) = (X+v)(I_r+v^\top v)^{-1/2}.\] For the polar retraction, we also have some nice properties summarized in the following lemma.
	\begin{lemma}\label{lem-1}
		\cite{BoumalNicolas} Let $R$ denote a retraction on $\text{St}(n,r)$. Then there exists a constant $M$ such that
		\begin{equation}\label{Rx}
		\begin{aligned}
		\|R_{X}(v) -(X+v)\|_F \leq M\|v\|_F^2,\\
		\forall X\in \text{St}(n,r),\ \forall v \in T_{X}\mathcal{M}.
		\end{aligned}
		\end{equation}
		Moreover, if the retraction is a polar retraction, then for any $X, Y \in \text{St}(n,r)$ and $v \in T_X\mathcal{M}$, it holds \cite[Lemma 1]{doi:10.1137/20M1321000}:
		\[\|\mathcal{R}_{X}(v) -Y\|_F^2 \leq \|X+v -Y\|_F^2.\]
	\end{lemma}
	\begin{remark}
		The inequality \eqref{Rx} implies that $R_X$ satisfies $R_X(v) = X+v+\mathcal{O}(\|v\|_F^2)$. Specifically, for the polar retraction, we also have a non-expansiveness property. Note that the lemma only applies to the compact submanifolds of the Euclidean space. It has been shown in the appendix of \cite{BoumalNicolas} that the constant $M$ can be calculated by \[M = \max \left\{\frac{1}{2} \max_{\zeta \in K}\|D^2 R_X(\zeta)\|, \frac{\text{diam}(\mathcal{M})+r}{r^2}\|v\|^2\right\} ,\]  where retraction $R$ is smooth on the set $K:=\{\zeta \in T\mathcal{M}:\|\zeta\| \leq r\}$, and $\text{diam}(\mathcal{M})$ is the diameter of $\mathcal{M}$ (which is finite as $\mathcal{M}$ is compact). \cite{liu2019quadratic} further shows that for polar retraction, if $\|v\|_F \leq 1$ then $M = 1$.
	\end{remark}
	
	\subsection{Average on Stiefel Manifold}\label{sec-2.2}
	We denote the variable of each agent $i$ by $X_i$. By stacking all of them, we have $\mathbf{X}:=[X_1^\top,\dots,X_N^\top]^\top \in \text{St}(n,r)^N$. Denote the Euclidean average of the variables by $\bar X:=\sum_{i=1}^N X_i $ and $ \bold{\bar X}: = (\bm{1}_N \otimes I_n) \bar X$.
	
	
	In Euclidean space, the multi-agent systems ultimately achieve consensus to the Euclidean average of all the agents, and the error bound $\|X_i-\bar X\|_F$ is typically used. When the variables satisfy $X_1, \dots,  X_N \in \text{St}(n,r)$, their Euclidean average may not be on the Stiefel manifold. Thus, we extend the concept of average in Euclidean space to the Stiefel manifold and introduce the induced arithmetic mean (IAM) \cite{doi:10.1137/060673400}, which is defined as  \begin{equation}\label{iam}
	\hat X \in \arg \min_{Y \in \text{St}(n,r)} \sum_{i=1}^N \|Y - X_i\|_F^2.
	\end{equation}
	
It has been shown in \cite{doi:10.1137/060673400} that IAM is the orthogonal projection of the Euclidean average onto the Stiefel manifold, i.e., \[\hat X= \mathcal{P}_{\text{St}}(\bar X).\]  Let $\mathcal{X}^*:=\{\bold{X} \in \text{St}(n,r)^N:X_1 = X_2 = \dots =X_N\}$ denote the consensus configuration on $\text{St}(n,r)$ \cite{doi:10.1137/060673400}. Then the distance from $\bold{X}$ to $\mathcal{X}^*$ is given by\[\text{dist} ^2 (\bold{X}, \mathcal{X}^*) =  \inf_{Y \in \text{St}(n,r)} \frac{1}{N} \sum_{i=1}^N \|Y-X_i\|_F^2 \overset{\eqref{iam}}{=} \frac{1}{N}\|\bold{X}- \bold{\hat X}\|_F^2,\] which will be adopted as a metric to analyze the consensus error in the later sections.

	\begin{remark}
		Here, we consider the consensus to IAM as it is essentially similar to the Karcher mean, which has a wide range of applications in many machine learning problems, for instance, the solution of PCA is the Karcher mean of the data samples. The difference between Karcher mean and IAM lies in that they use geodesic distance and Euclidean distance in the definition, respectively. In this context, we use the IAM due to its convenience in computation.
	\end{remark}
	
	\subsection{Optimality Condition}\label{sec-2.3}
	In this part, we will  introduce the optimality condition regarding the optimization problem on the manifold. First of all, we need to state the definition of Riemannian gradient since it is vital important in searching the optimal solution.  For a function $\phi$ defined on the Stiefel manifold, the only tangent vector that satisfies $\langle \text{grad}\  \phi(X), v \rangle = D \phi(X)[v]$ ($\forall v \in T_X \mathcal{M}$) is called a Riemannian gradient $\text{grad}\  \phi(X)$, where $D \phi(X)[v] = \frac{d f(\gamma(t))}{dt} \big|_{t=0}$ and the curve $\gamma$ satisfies $\gamma(0) = X$.  As the Stiefel manifold is embedded on the Euclidean space, we have the following relation between Euclidean and Riemannian gradient \cite{AbsilMahonySepulchre+2008}: \begin{equation}\label{rg}
	\text{grad}\ \phi(X) = \mathcal{P}_{T_X \mathcal{M}}(\nabla \phi(X)),
	\end{equation}where $\mathcal{P}_{T_X\mathcal{M}}$ denote the orthogonal projection onto the tangent space $T_X\mathcal{M}$. Especially, for $\forall Y\in \mathbb{R}^{n \times r}$, we have: \begin{equation}\label{1}
	\mathcal{P}_{T_X\mathcal{M}}(Y) =Y-\frac{1}{2}X(X^\top Y+ Y^\top X).
	\end{equation}
	
	Lipschitz smooth is a typically used concept in theoretical analysis of optimization problems \cite{qu2017harnessing, pu2021distributed}. Generally, a function $\phi$ is said to be $L$-smooth, if $\phi$ is differentiable and has $L$-Lipschitz continuous Euclidean gradient, i.e., for $\forall X,Y \in \mathbb{R}^{n \times r}$ \[\|\nabla \phi(X) - \nabla \phi(Y)\|_F \leq L\|X-Y\|_F.\]
	
	We then provide a Lipschitz-type inequality for the functions on the Stiefel manifold and formally state it in the following lemma.
	\begin{lemma}\label{al-4}
		\cite[Lemma 2.4]{pmlr-v139-chen21g} For any $X,Y\in \text{St}(n,r)$, if a function $\phi(X)$ is L-smooth in Euclidean space, then there exists a constant $L_g=L+L_n$ with $L_n = \max_{X \in \text{St}(n,r)} \|\nabla \phi(X)\|_2$ such that \begin{equation}\label{smooth}
		\left|\phi(Y)-[\phi(X)+\langle \text{grad}\ \phi(X),Y-X\rangle]\right| \leq \frac{L_g}{2}\|Y-X\|_F^2.
		\end{equation}
		Moreover, if $L_G = L+2L_n$, it holds
		\begin{equation}
		\|\text{grad}\ \phi(X)-\text{grad}\ \phi(Y)\|_F\leq L_G\|X-Y\|_F.
		\end{equation}
	\end{lemma}
	
	Next, we give the necessary first-order optimality condition of problem \eqref{p-1}.
	\begin{proposition}
		\cite{BoumalNicolas}Suppose $f$ is differentiable at $X$, if $X \in \mathcal{M}$ is a local optimum for problem \eqref{p-1},  then it holds $\text{grad}f(X) = 0$.
	\end{proposition}
	
	\section{Distributed Riemannian Stochastic Gradient Tracking Method}\label{sec-3}
	In this section, we first reformulate the problem \eqref{p-1}, and then propose a distributed Riemannian stochastic optimization algorithm for solving it based on the variable sampling scheme and gradient tracking method.
	
	\subsection{Problem Reformulation}\label{sec-2.4}
	Let $X_i \in \text{St}(n,r)$ be a local copy of the optimization variable of each agent $i$. Since the graph $\mathcal{G}$ is connected, $X_i=X_j,~\forall  (i,j) \in \mathcal{E}$ is equivalent to the consensus condition $X_1=\cdots=X_N$. We reformulate the optimization problem \eqref{p-1} as
	\begin{subequations}\label{p-2}
		\begin{align}
		\min&\ \frac{1}{N}\sum_{i=1}^Nf_i(X_i),\\
		s.t.~& X_i=X_j,~\forall (i,j) \in \mathcal{E},\label{c-1}\\
		&X_i \in \text{St}(n,r),\ i\in \mathcal{N}.
		\end{align}
	\end{subequations}
	The local cost function for agent $i$ is defined as $f_i(X_i): = \mathbb{E}_{\xi_i } F_i(X_i, \xi_i)$, where $\xi_i$ is a random variable defined on the probability space $(\Omega_i, \mathcal{F}_i, \mathbb{P})$ and usually represents a data sample in machine learning. Many popular machine learning models can be formulated by this optimization problem, including logistic regression, dictionary learning \cite{7219480}, and deep learning \cite{lecun2015deep}. The index $i$ indicates that the accessible data set for each agent might be different.
	
	We assume that the communication graph $\mathcal{G}=(\mathcal{N},\mathcal{E})$  is undirected,  namely, $(j,i)\in \mathcal{E}$ if and only if $(i,j)\in \mathcal{E}$. We then define an associated   adjacency matrix $W:=\{w_{ij}\}_{N\times N}$, where $w_{ij}>0$ if $(i,j)\in \mathcal{E}$ or $i=j$, and $w_{ij}=0$, otherwise. We make the following assumptions on $\mathcal{G}$, which are standard in the existing literatures.
	\begin{assumption}\label{assu-1}
		Suppose the graph $\mathcal{G}$ is undirected and connected, and $W$ satisfies
		\begin{itemize}
			\item[1)]$W=W^\top$;
			\item[2)]$W$ is nonnegative and $W \bm{1}_N = W^\top \bm{1}_N = \bm{1}_N$ .
		\end{itemize}
	\end{assumption}
	\begin{remark}\label{r-1}
		According to Assumption \ref{assu-1}, any power of the matrix $W$, i.e., $W^t(=W^{t-1}W)$, where $t \geq 1$, is also symmetric and doubly stochastic. Since the graph is connected, all the eigenvalues of $W$ are in $(-1,1]$, and the second largest singular value $\sigma_2\in [0,1)$ \cite{1406483}, which is also the spectral norm of $W - \frac{1}{N}\bm{1_N} \bm{1_N}^\top$ \cite[Th. 5.1]{doi:10.1137/060678324}.
	\end{remark}
	
	Since the distribution of each $\xi_i$ is assumed to be unknown, the expectation-valued local cost functions of the agents cannot be calculated explicitly. To solve problem \eqref{p-2}, we assume there exists a stochastic oracle to obtain noisy Riemannian gradient samples of the form $G_i(X, \xi_i)$. To articulate sufficiency conditions, we make some assumptions on the cost functions and the stochastic oracle.
	\begin{assumption}\label{assu-5}
		For each agent $i$, the local cost function $f_i$ is L-smooth in Euclidean space.
	\end{assumption}
	
	We impose the following Assumption \ref{assu-2} and \ref{assu-3} to better estimate the exact Riemannian gradient. These assumptions hold for many online distributed optimization problems and can also be seen in \cite{pmlr-v139-chen21g, 6487381}.
	\begin{assumption}\label{assu-2}
		The Riemannian gradient sample is unbiased, i.e., for each $i \in \mathcal{N}$ and any given $X \in \text{St}(n,r)$, $E_{\xi_i}[G(X, \xi_i)] = \text{grad}f_i(X)$; and has bounded variance, i.e.,  $E_{\xi_i}[\|G(X, \xi_i) - \text{grad}f_i(X)\|_F^2| X] \leq \sigma^2$.
	\end{assumption}
	
	\begin{assumption}\label{assu-3}
		The norm of the noisy Riemannian gradient samples are uniformly bounded, i.e., for all $i \in \mathcal{N}$, there exists a constant $A_1>0$, such that $\max_{X\in \mathcal{M},\xi_i \in \text{supp}(\mathbb{P})}\|G_i(X, \xi_{i})\| \leq A_1$.
	\end{assumption}
	\begin{remark}\label{re-4}
		Since $f$ is the average of all $f_i$, Assumption \ref{assu-5} implies that $f$ is also L-smooth.
		Due to the compactness of the Stiefel manifold, there exists a constant $A_2$ such that $\max_{X \in \text{St}(n,r)}\|\nabla f_i(X)\|_F \leq A_2$ according to Assumption \ref{assu-5}. Thus, we have $\max_{X \in \text{St}(n,r)}\|\text{grad} f_i(X)\|_F \leq A_2$ by \eqref{rg} and the nonexpansiveness of the orthogonal projection. For the sake of simplicity in writing, we denote $A:= \max\{A_1, A_2\}$.
	\end{remark}
	
	\subsection{Algorithm Design}
	Inspired by the distributed Euclidean algorithms in \cite{pu2021distributed,qu2017harnessing}, the idea of finding the optimal solution of problem \eqref{p-2} is based on a consensus dynamic and a gradient descent dynamic.
	
	First, we briefly introduce the consensus problem on the Stiefel manifold. Let $h_{i,t}(\bold{X}):=\frac{1}{2} \sum_{j=1}^N W_{ij}^t\|X_i-X_j\|_F^2$ denote the local consensus potential. As discussed in \cite{MARKDAHL2020108736}, the problem can be formulated as follows.
	\begin{equation}\label{p-3}
	\min_{X_i \in \text{St(n,r)}} h_t(\bold{X}):=\frac{1}{2}\sum_{i=1}^Nh_{i,t}(\bold{X})=\frac{1}{4}\sum_{i=1}^N \sum_{j=1}^N W_{ij}^t\|X_i-X_j\|_F^2,
	\end{equation}
	where $W_{ij}^t$ represent the $i$-$j$th element of $W^t$. To solve problem \eqref{p-3}, each agent $i \in \mathcal{N}$ communicates with its direct neighbors and computes the weighted average for $t$ times in one iteration by the term $\sum_{j=1}^N W_{ij}^t X_{j}$. Since this may not stays on the tangent space $T_{X_{i}} \mathcal{M}$, a projection step is performed to utilize the nice property of the retraction.
	
	Let $X_{i,k}$ denote the estimation of optimal solution to the problem \eqref{p-1} at iteration $k$. To further solve problem \eqref{p-1}, we can view it as the following Riemannian optimization problem:
	\[	\min_{\bold{X}\in \text{St}^N(n,r)} \ \alpha h_t(\bold{X})+ \beta f(\bold{X}),\] which gives us some insight into combining the consensus dynamic \cite{chen2021local} and an iteration direction $v_{i,k}$ to obtain the total search direction. We also use the retraction to ensure feasibility on the Stiefel manifold. Hence, each agent $i$ updates variable $X_{i,k}$ by: \begin{equation}\label{4}
	X_{i,k+1} = \mathcal{R}_{X_{i,k}} \left(\alpha \mathcal{P}_{T_{X_{i,k}} \mathcal{M}}\left[\sum_{j=1}^N
	W_{ij}^t X_{j,k}\right] - \beta v_{i,k}\right),
	\end{equation}
	where $v_{i,k}$ is the iteration direction later defined in \eqref{v}. We replace the commonly used exponential map \cite{shah2017distributed,zhang2016first} by retraction to present an efficient algorithm, as has been discussed in Section \ref{sec-2.1}.
	
	As the distribution of each $\xi_i$ is unknown, the cost function $f_i(X)$ usually has no explicit expression which leads to a lack of the exact $\text{grad} f_i(X)$. To resolve this, we usually use stochastic approximation in algorithm design. For any given $X, \xi$, there is a first-order stochastic oracle that returns some noisy Riemannian gradient samples of the form $G_i(X, \xi) \in T_X \mathcal{M}$, which is an unbiased estimation of the Riemannian gradient of $f_i$. Denote the average of variable sampled gradients by\begin{equation}\label{v-s}
	F_i(X_{i,k+1}) = \frac{1}{N_k} \sum_{j=1}^{N_k} G_i(X_{i,k+1}, \xi_{i,k+1}^{[j]}),
	\end{equation} where $N_k$ is the number of samples used at iteration $k$, and $\xi_{i,k}^{[j]}$($\forall j=1,\dots, N_k$) are the i.i.d. realizations of the random variable $\xi_{i,k}$. By appropriately increasing the size of sampling with iteration, we can reduce the stochastic variance caused by noise. As the samples $\xi_{i,k}^{[j]} $ are independent for $j = 1 ,\dots, N_k$, we derive  from Assumption \ref{assu-2} that for any $k\geq 0$:
	\begin{equation}\label{oracle}
	\begin{aligned}
	&E_{\xi_i}[F_i(X_{i,k})] = \frac{1 }{N_k}\sum_{j=1}^{N_k} E[G_i(X_{i,k}, \xi_{i,k}^{[j]})] =\text{grad}f_i(X_{i,k}),\\
	&E[\|F_i(X_{i,k})-\text{grad}f_i(X_{i,k})\|^2| X_{i,k}] \\
	=& \frac{\sum_{j=1}^{N_k} E[\|G_i(X_{i,k}, \xi_{i,k}^{[j]})- \text{grad}f_i(X_{i,k})\|^2|X_{i,k}] }{N_k^2} =\frac{ \sigma ^2}{N_k}.
	\end{aligned}
	\end{equation}
	This implies that $F_i(X_{i,k})$ is also an unbiased estimation, and its variance decreases with the increasing of sample size $N_k$.
	\begin{remark}
		Though the authors of \cite{lei2022distributed} have proposed a distributed algorithm based on the idea of variable sampling in Euclidean space, our Algorithm \ref{alg:1} is still worth analyzing. It is usually invalid to add the tangent vectors belonging to different tangent spaces directly. Hopefully, \eqref{v-s} is well-defined on the Stiefel manifold, since $G_i(X, \xi)$ are actually matrices, for any $X, \xi$. As have been stated in Introduction, such a distributed Riemannian stochastic gradient tracking algorithm has not been analyzed on the manifold yet.
	\end{remark}
	
	For each agent $i \in \mathcal{N}$, we introduce an auxiliary variable $Y_{i,k}$ to asymptotically track the dynamical average gradient across the network, which is updated by \begin{equation}\label{g-t}
	Y_{i,k+1} = \sum_{j=1}^{N}W_{ij}^t Y_{j,k}+F_i(X_{i,k+1})-F_i(X_{i,k}).
	\end{equation} Since the Stiefel manifold is the embedded submanifold of Euclidean space, $F_i(X_{i,k})$ can be viewed as the projected Euclidean gradient, it is free to calculate $F_i(X_{i,k})-F_i(X_{i,k-1})$. However, $Y_{i,k}$ not necessarily remain on the tangent space $T_{X_{i,k}}\mathcal{M}$ after performing the gradient tracking step. In order to use the nice property of the retraction,  it is important to take an orthogonal projection onto the tangent space before updating \eqref{4}, namely, \begin{equation}\label{v}
		v_{i,k} = \mathcal{P}_{T_{X_{i,k}} \mathcal{M}} Y_{i,k}.
	\end{equation}
	
	We summarized the procedure of our algorithm as follows.
	\begin{algorithm}
		\caption{Distributed Riemannian Stochastic Gradient Tracking Algorithm}
		\label{alg:1}
		\begin{algorithmic}
			\REQUIRE For all $i\in \mathcal{N}$, set $\bold{X}_{0} \in \mathcal{S},\ Y_{i,0} = F_i(X_{i,0})$, where the local region $\mathcal{S}$ will be formally defined in Definition \ref{l-r}. Let the constant step sizes $\alpha, \beta>0$ and $t \geq \left[\log_{\sigma_2}\left(\frac{1}{2 \sqrt{N}}\right)\right]$.
			\FOR {$k=0,1,2,\dots,K$}
			\STATE Project $Y_{i,k}$ onto the tangent space by \eqref{v};
			\STATE Update the optimal variable $X_{i,k+1}$by \eqref{4};
			\STATE Sample the stochastic gradient by \eqref{v-s};
			\STATE Update the auxiliary variable $Y_{i,k+1}$ by \eqref{g-t}.
			\ENDFOR
		\end{algorithmic}
	\end{algorithm}
	
	\section{Main Results}\label{sec-4}
	We analyze the convergence and establish the convergence rate of Algorithm \ref{alg:1} in this section.

\subsection{Technical Lemmas}

According to \cite{chen2021local}, the multi-step consensus algorithm on the Stiefel manifold achieves Q-linearly convergence in a local region due to its non-convexity. Thus, the convergence analysis of Algorithm \ref{alg:1} is restricted to this local region in the later context. Next, we formally introduce its definition.
\begin{definition}\label{l-r}
	For given constants $\delta_1 \leq \frac{1}{5\sqrt{r}}\delta_2$, and $\delta_2 \leq \frac{1}{6}$, the local region is defined as
	$$\begin{aligned}
	\mathcal{S}&:=\mathcal{S}_1 \bigcap\mathcal{S}_2,\\
	\mathcal{S}_1&:=\{\bold{X}:\|\bold{X}-\hat {\bold{X}}\|_F^2 \leq N \delta_1^2\}, \\
	\mathcal{S}_2&:=\{\bold{X}: \bold{X}\in \text{St}(n,r)^N, \max_{i=1, \dots, N}\|X_i - \hat X\|_F \leq \delta_2\}.
	\end{aligned}$$
\end{definition}

Due to the lack of explicit expressions for IAM, which is difficult to quantitatively estimate and calculate. Therefore, we introduce some technical lemmas that will be used in the convergence analysis.

First, we show that the distance between $\bar X$ and $\hat X$ are bounded by the consensus error.
\begin{lemma}\label{al-1}
	\cite[Lem.1]{chen2021local} For any $ \bold{X} \in \text{St}(n,r)^N$, let $\hat X$ denote its IAM (defined by \eqref{iam}) and $\bold{\hat X}: = (\bm{1}_N \otimes I_n) \hat X$. If $\|\bold{X}- \bold{\hat X}\|_F^2 \leq N/2$, it holds
	\[\|\hat X-\bar X\|_F \leq \frac{2\sqrt{r}\|\bold{X}- \bold{\hat X}\|_F^2}{N}.\]
\end{lemma}

Suppose that $\hat X,\hat Y,\bar X,\bar Y$ are defined as in Section \ref{sec-2.2}. Then $\hat X,\hat Y$ can be viewed as the polar decomposition of $\bar X,\bar Y$ \cite{doi:10.1137/S0895479801394623}, respectively. The following lemma helps to estimate the bound of the distance between $\hat X_{k+1}$ and $\hat X_{k}$.
\begin{lemma}\label{pr}
	\cite[Th.2.4]{doi:10.1137/S0895479801394623} Denote the eigenvalues of $\bar{X}\in \text{St}(n,r)$ are $\sigma_1(\bar{X})\geq \sigma_2(\bar{X})\geq \dots \geq \sigma_r(\bar{X})$, likely wise, the eigenvalues of $\bar{Y}\in \text{St}(n,r)$ are $\sigma_1(\bar{Y})\geq \sigma_2(\bar{Y})\geq \dots \geq \sigma_r(\bar{Y})$. It holds
	\begin{equation}\label{al-2}
	\|\hat Y-\hat X\|_F \leq \frac{2}{\sigma_r(\bar X)+\sigma_r(\bar Y)}\|\bar Y-\bar X\|_F.
	\end{equation}
	Moreover, if we also have $\bold{X}, \bold{Y} \in \mathcal{S}_1$, then it holds \cite[Lem.12]{chen2021local}
	\begin{equation}\label{col-1}
	\frac{2}{\sigma_r(\bar X)+\sigma_r(\bar Y)} \leq \frac{1}{1-2\delta_1^2},
	\end{equation}
	where $\delta_1$ is defined in Definition \ref{l-r}.
\end{lemma}

We also present the following lemma to estimate $\text{grad}\ h_t(\bold{X})$, where $  h_t(\bold{X})$ is defined in \eqref{p-3}.
\begin{lemma}\label{al-3}
	\cite[Lem.10]{chen2021local} For any $\bold{X} \in \text{St}(n,r)^N$, we have
	\begin{equation}\label{9-1}
	\|\sum_{i=1}^N \text{grad} \ h_{i,t}(\bold{X})\|_F \leq L_t\|\bold{X}- \bold{\hat X}\|_F^2,
	\end{equation}
	and \begin{equation}\label{9-2}
	\|\text{grad}\  h_t(\bold{X})\|_F\leq L_t	\|\bold{X}- \bold{\hat X}\|_F,
	\end{equation}
	where $L_t:=1-\sigma_N(W^t)$, and $\sigma_N(W^t)$ represents the smallest eigenvalue of $W^t$. Then it holds $L_t\in (0,2]$. Especially, if $\bold{X} \in \mathcal{S}_2$, we also have \begin{align*}
	\max_{i \in \mathcal{N}}\|\text{grad} \ h_{i,t}(\bold{X})\|_F \leq 2 \delta_2.
	\end{align*}
	where $\delta_2$ is defined in Definition \ref{l-r}.
\end{lemma}

\begin{lemma}\label{al-6}
	\cite[Prop.4]{chen2021local} For any $\bold{X} = [X_1^\top, \dots, X_N^\top]^\top \in \text{St}(n,r)^N$, if $\bold{X} \in \mathcal{S}$, then the following inequality holds.
	\[\langle \bold{X}-\bold{\hat X}, \text{grad}\ h_t(\bold{X}) \rangle \geq \frac{\Phi}{2L_t} \|\text{grad}\ h_t(\bold{X})\|_F^2,\]
	where $\Phi \geq 1$ is a constant related to $\bold{X}$, which is defined as $\Phi:=2-\max_{i=1, \dots, N}\|X_i - \hat X\|_F^2$.
\end{lemma}

In the following, we will introduce an iterative property for the proof of convergence.
\begin{lemma}\label{al-5}
	\cite[Lem.2]{7402509} Denote $\{u(k)\}_{k\geq 0}$ and $\{w(k)\}_{k \geq 0}$ as positive scalar sequences. Suppose that for any  $k \geq 0$, it holds $$u(k+1) \leq \gamma u(k)+w(k), $$where the parameter $\gamma \in(0,1)$. Let $\mathcal{U}(k)=\sum_{l=0}^k u(l)$ and $\Omega(k)=\sum_{l=0}^k w(l)$. Then it satisfies $$\mathcal{U}(k) \leq a \Omega(k)+b,$$where $a = \frac{2}{(1-\gamma)^2}$ and $b = \frac{2}{1-\gamma^2}u(0)$.
\end{lemma}

\subsection{Preliminary Results}
In this part, we   establish  some preliminary  results that will be used in the convergence and rate analysis of Algorithm \ref{alg:1}.

Define the vectorized form of the variable-sampled stochastic gradient in \eqref{v-s}, and its average among all agents as
	\begin{align*}
		&F(\bold{X}_k) = [F_1(X_{1,k})^\top, F_2(X_{2,k})^\top, \cdots, F_N(X_{N,k})^\top]^\top, \\
		&\bar F_k := \frac{1}{N}\sum_{i=1}^N F_i(X_{i,k}),\bar{\bm{F}}_k:=(\bm{1}_N \otimes I_n) \bar F_k.
	\end{align*}
	Denote the averaged gradient and the averaged estimate of the gradient across the network by
	\begin{align*}
	&g(\bold{X}_k) = [\text{grad}f_1(X_{1,k})^\top, \text{grad} f_2(X_{2,k})^\top, \cdots, \text{grad} f_N(X_{N,k})^\top]^\top, \\
	&\bar g_k := \frac{1}{N}\sum_{i=1}^N \text{grad} f_i(X_{i,k});\\
	&\bold{Y}_k = [Y_{1,k}^\top,Y_{2,k}^\top, \cdots, Y_{N,k}^\top]^\top,  \bar Y_k := \frac{1}{N}\sum_{i=1}^N Y_{i,k};\\
	&\bold{v}_{k} =[v_{1,k}^\top,\dots,v_{N,k}^\top]^\top ,  \bar v_k := \frac{1}{N}\sum_{i=1}^N v_{i,k}.
	\end{align*}
	
	According to \eqref{rg}, we have $-\text{grad}\ h_{i,t}(\bold{X}_k)= - \mathcal{P}_{T_{X_{i,k}} \mathcal{M}}(\nabla h_{i,t}(\bold{X}_k)) = \mathcal{P}_{T_{X_{i,k}} \mathcal{M}}(\sum_{j=1}^N W_{ij}^t X_{j,k})$. Consider the update step \eqref{4} in Algorithm \ref{alg:1}, which can be rewritten as \begin{equation}\label{re-u}
	X_{i,k+1} = \mathcal{R}_{X_{i,k}}(-\alpha \text{grad}\ h_{i,t}(\bold{X}_k)-\beta v_{i,k}).
	\end{equation}
	
	Similar to the gradient tracking algorithm \cite{qu2017harnessing} in Euclidean space, we also have the following result on the Stiefel manifold, which shows that $\{\bold{Y}_k\}$ tracks the average of sampled gradients.
	\begin{lemma}\label{gt}
		Suppose Assumption \ref{assu-1} holds, we have
		$$\frac{1}{N}\sum_{i=1}^N Y_{i,k} = \frac{1}{N}\sum_{i=1}^N F_i(X_{i,k}),\ i.e.,  \ \bar Y_k = \bar F_k.$$
	\end{lemma}
	
	Next, we give the iteration relations of gradient tracking error and consensus error. The proof is given in \ref{app_1}.
	\begin{lemma}\label{lem-5}
		Suppose Assumptions \ref{assu-1}-\ref{assu-5} hold. Consider Algorithm \ref{alg:1}, where the step size  satisfies $\alpha \leq \frac{\Phi}{L_t}$ with $\Phi $ and $L_t$ defined in Lemma \ref{al-6} and Lemma \ref{al-3}, respectively. Then we can derive the following results.
		
		(i) Accumulated gradient tracking error:
			\begin{align}\label{l-2}
			&\frac{1}{N}\sum_{k=0}^KE[\|\bold{Y}_k-\bar{\bold{F}}_k\|_F^2|\mathcal{F}_k]\notag\\
			\leq&\Gamma_2\frac{1}{N}\sum_{k=0}^KE[\|F(\bold{X}_{k+1})-F(\bold{X}_k)\|_F^2|\mathcal{F}_k]+\Gamma_3,
			\end{align}
			where $\Gamma_2 = \frac{2}{(1-\sigma_2^t)^2}$, and $\Gamma_3 = \frac{2}{1-\sigma_2^{2t}}\cdot \frac{1}{n}\|\bold{Y}_0 - \bar{\bold{F}}_0\|_F^2$, here $\sigma_2$ is the second largest singular value of $W$ according to Remark \ref{r-1};
		
		(ii) Consensus error: If $\bold{X}_k \in \mathcal{S}$, then
			\begin{equation}\label{l-3}
			\|\bold{X}_{k+1}-\bold{\hat X}_{k+1}\|_F \leq \rho_t \|\bold{X}_k-\bold{\hat X}_k\|_F+\beta\|\bold{Y}_k\|_F,
			\end{equation}
			where $\rho_t:=\sqrt{1-(\alpha L_t \Phi - \alpha^2 L_t^2)}$. Moreover, we have the following inequality with $\Gamma_0 = \frac{2}{(1-\rho_t)^2}$, and $\Gamma_1 = \frac{2}{1-\rho_t^2}\|\bold{X}_0-\bold{\hat X}_0\|_F^2$:
			\begin{equation}\label{l-4}
			\frac{1}{N}\sum_{k=0}^K\|\bold{X}_k-\bold{\hat X}_k\|_F^2 \leq \Gamma_0 \frac{1}{N}\beta^2\sum_{k=0}^K \|\bold{Y}_k\|_F^2+\Gamma_1.
			\end{equation}
	\end{lemma}

	The following lemma shows  that the iterates $\bold{X}_{k}$ will always stay in the set $\mathcal{S}$ when the initial points start in $\mathcal{S}$, for which the  proof is given in \ref{app_2}.
	\begin{lemma}\label{lem-3}
		Suppose Assumptions \ref{assu-1}-\ref{assu-3} hold. 
Consider Algorithm \ref{alg:1},  where $t \geq [\log_{\sigma_2}(\frac{1}{2\sqrt{N}})]$,  the step sizes $\alpha$ satisfies $0< \alpha \leq \min\{1,\frac{1}{M},\frac{\Phi}{L_t}\}$ with $\Phi $ and $L_t$ defined in Lemma \ref{al-6} and Lemma \ref{al-3}, and $\beta$ satisfies $0 \leq \beta \leq \bar \beta:= \min\{\frac{1-\rho_t}{10A+L_G}\delta_1,\frac{\alpha \delta_1}{5(10A+L_G)}\}$ with $L_G$ defined in Lemma \ref{al-4}. If the initial point  satisfies $\bold{X}_0 \in \mathcal{S}$, then for $\forall k \geq 0$, we have
		\begin{itemize}
			\item[(i)] $\bold{X}_k \in \mathcal{S}$;
			\item[(ii)] $Y_{i,k}$ is uniformly bounded, i.e., $\|Y_{i,k}\|_F \leq 10A+L_G$, $\forall i \in \mathcal{N}$.
		\end{itemize}
	\end{lemma}

	Before giving the convergence theorem, we introduce a sufficiently descent lemma in advance, which shows the relationship between the function values at $k$ and $k+1$.
	\begin{lemma}\label{lem-4}
		Suppose Assumption \ref{assu-1}-\ref{assu-3} hold. Consider Algorithm \ref{alg:1},  where  $\bold{X}_0 \in \mathcal{S}$  and the multi-steps of communication $t \geq [\log_{\sigma_2}(\frac{1}{2\sqrt{N}})]$. The step size $\alpha$ satisfies $0< \alpha \leq \min\{1,\frac{1}{M},\frac{\Phi}{L_t}\}$, and step size $\beta$ satisfies $0 \leq \beta \leq \min\{\bar \beta,\frac{1}{9L_G}\},$ where $\bar \beta$ is given in Lemma \ref{lem-3}.  Denote by $\hat X_{k+1}$ the IAM in \eqref{iam} and by $\bold{\hat X}: = (\bm{1}_N \otimes I_n) \hat X$. Then we have
		\begin{equation*}
		\begin{aligned}
		E[f(\hat X_{k+1})|\mathcal{F}_k] \leq & f(\hat X_k)-\frac{\beta}{2}\|\bar g_k\|_F^2+\frac{D_1}{N}\beta^2E[\|\bold{Y}_k\|_F^2|\mathcal{F}_k]\\
		+&\frac{D_2}{N}\|\bold{X}_k-\bold{\hat X}_k\|_F^2+\frac{\sigma^2D_3}{ N_k}+\frac{D_4}{N}\|\bold{ X}_{k+1}-\hat{\bold{X}}_{k+1}\|_F^2,
		\end{aligned}
		\end{equation*}
		where $D_1 = (\delta_1^2+5)L_G+2MA$ with $M$ given by Lemma \ref{lem-1}, $ D_3 = \frac{1+4 L_G^2 \beta^2}{2L_G},$ $D_2 = 3L_G+2\alpha A+13L_G\tilde C \alpha^2+D_4$ with $\tilde C = \frac{\delta_2^2}{r }$, $L_G$ given by Lemma \ref{al-4}, $A$ defined in Remark 4, and $ D_4 = \frac{4r \delta_1^2}{\beta^2L_G}$.
	\end{lemma}
	\begin{proof}
		Denote the conditional expectation $E[\cdot | \mathcal{F}_k]$, where the sigma-algebra is defined as $\mathcal{F}_k:=\sigma \{x_{i,0},  \xi_{i,t-1}^{[j]}, 0 \leq j \leq N_t, 0 \leq t \leq k-1, i=1, \dots, N\}$. Then $\bold{X}_k$ is adapted to $\mathcal{F}_k$, and hence $\hat X_k\in \mathcal{F}_k$. By using \eqref{smooth} and noticing that $L_g \leq L_G$ in Lemma \ref{al-4}, we have
		\begin{align}\label{3-1}
		&E[f(\hat X_{k+1})|\mathcal{F}_k] \notag\\
		\leq & f(\hat X_k)+\langle \text{grad}f(\hat X_k), E[\hat X_{k+1} - \hat X_{k} | \mathcal{F}_k]\rangle+\frac{L_G}{2}E[\|\hat X_{k+1} - \hat X_{k}\|_F^2|\mathcal{F}_k]\notag\\
		\leq & f(\hat X_k)+\langle \text{grad}f(\hat X_k) - \bar g_k, E[\hat X_{k+1} - \hat X_{k}| \mathcal{F}_k]\rangle\notag\\
		+&\langle \bar g_k, E[\hat X_{k+1} - \hat X_{k}| \mathcal{F}_k]\rangle+\frac{L_G}{2}E[\|\hat X_{k+1} - \hat X_{k}\|_F^2|\mathcal{F}_k]\\
		\leq & f(\hat X_k) + \frac{1}{L_G}\|\text{grad}f(\hat X_k) - \bar g_k\|_F^2+\frac{3L_G}{4}E[\|\hat X_{k+1} - \hat X_{k}\|_F^2|\mathcal{F}_k]\notag\\
		+&\langle \bar g_k, E[\hat X_{k+1} - \hat X_{k}| \mathcal{F}_k]\rangle, \notag
		\end{align}
		where the last inequality follows by the Young's inequality.
		Furthermore, we also have
		\begin{align*}
		&\langle \bar g_k, E[\hat X_{k+1} - \hat X_{k}| \mathcal{F}_k]\rangle \leq \langle \bar g_k,E[\bar{X}_{k+1} - \bar{X}_{k}| \mathcal{F}_k]\rangle +\frac{\beta^2 L_G}{2}\|\bar g_k\|_F^2\\
		+&\frac{1}{2\beta^2 L_G}E[\|\hat X_{k+1}-\bar{X}_{k+1} + \hat X_{k}- \bar{X}_{k}\|_F^2| \mathcal{F}_k].
		\end{align*}

		This combined with \eqref{3-1} produces
		\begin{equation}\label{3-3}
		\begin{aligned}
		E[f(\hat X_{k+1})|\mathcal{F}_k]\leq & f(\hat X_k) + \frac{1}{L_G}\|\text{grad}f(\hat X_k) - \bar g_k\|_F^2\\
		+& \frac{1}{2\beta^2 L_G}E[\|\hat X_{k+1}-\bar{X}_{k+1} + \hat X_{k}- \bar{X}_{k}\|_F^2| \mathcal{F}_k]\\
		+&\frac{\beta^2 L_G}{2}\|\bar g_k\|_F^2+ \frac{3L_G}{4}E[\|\hat X_{k+1} - \hat X_{k}\|_F^2|\mathcal{F}_k]\\
		+&\langle \bar g_k, E[\bar{X}_{k+1} - \bar{X}_{k}| \mathcal{F}_k]\rangle.
		\end{aligned}
		\end{equation}
		By Assumption \ref{assu-2}, we can derive that $E[\bar F_k|\mathcal{F}_k] = \bar g_k$, which implies
			\begin{align*}
			&\langle \bar g_k, E[\bar{X}_{k+1} - \bar{X}_{k}| \mathcal{F}_k]\rangle \\
			= &\langle \bar g_k , E[\bar{X}_{k+1} - \bar{X}_{k}+\beta \bar F_k|\mathcal{F}_k] - \beta \bar g_k\rangle\\
			=&-\beta \|\bar g_k\|_F^2 +E[ \langle \bar F_k ,\frac{1}{N} \sum_{i=1}^N [{X}_{i,k+1} - {X}_{i,k}+\beta  F_i(X_{i,k})]\rangle
			\mathcal{F}_k ]
			\\	= &-\beta \|\bar g_k\|_F^2+E[ \langle \bar F_k , \frac{1}{N}\sum_{i=1}^N [\beta (Y_{i,k}-v_{i,k})-\alpha \text{grad}h_{i,t}(\bold{X}_k)] \rangle |
			\mathcal{F}_k]\\
			+& E[ \langle \bar F_k , \frac{1}{N} \sum_{i=1}^N ({X}_{i,k+1} - [{X}_{i,k}-\beta v_{i,k}-\alpha \text{grad}h_{i,t}(\bold{X}_k)]) \rangle |
			\mathcal{F}_k] ,
			\end{align*}
			where the last equality  utilized $\frac{1}{N}\sum_{i=1}^N Y_{i,k} = \frac{1}{N}\sum_{i=1}^N F_i(X_{i,k})$ from Lemma \ref{gt}.
		Substituting  the above bound into \eqref{3-3}, we have
		\begin{align}\label{main}
		&E[f(\hat X_{k+1})|\mathcal{F}_k]\notag\\
		\leq & f(\hat X_k) -(\beta - \frac{\beta^2 L_G}{2})\|\bar g_k\|_F^2+\frac{1}{L_G}\underbrace{\|\text{grad}f(\hat X_k) - \bar g_k\|_F^2}_{term\ 1}\notag\\
		+& \underbrace{\frac{1}{2\beta^2 L_G}E[\|\hat X_{k+1}-\bar{X}_{k+1} + \hat X_{k}- \bar{X}_{k}\|_F^2]}_{term\ 2}\notag\\
		+&\underbrace{E[ \langle \bar F_k , \frac{1}{N} \sum_{i=1}^N ({X}_{i,k+1} - [{X}_{i,k}-\beta v_{i,k}-\alpha \text{grad}h_{i,t}(\bold{X}_k)]) \rangle |
			\mathcal{F}_k]}_{term\ 3}\notag\\
		+&\underbrace{E[ \langle \bar F_k , \frac{1}{N}\sum_{i=1}^N [\beta (Y_{i,k}-v_{i,k})-\alpha \text{grad}h_{i,t}(\bold{X}_k)] \rangle |
			\mathcal{F}_k]}_{term\ 4} \notag\\
		+&\frac{3L_G}{4} \underbrace{ E[\|\hat X_{k+1} - \hat X_{k}\|_F^2|\mathcal{F}_k]}_{term\ 5}.
		\end{align}
		
		Here, we attempt to analyze the $term\ 1$ - $5$ in \eqref{main}. First, we focus on the $term\ 1$. Since Assumption \ref{assu-5} holds, and by using the Lipschitz-type inequality in Lemma \ref{al-4}, we derive that\begin{equation}\label{lip}
			\begin{aligned}
				term\ 1 &=\| \text{grad}f(\hat X_k)-\frac{1}{N} \sum_{i=1}^N \text{grad}f_i(X_{i,k})\|_F^2\\
				&\leq \frac{1}{N}\sum_{i=1}^N \|\text{grad}f_i(\hat X_k)-\text{grad}f_i(X_{i,k})\|_F^2\\
				& \leq \frac{L_G^2}{N}\sum_{i=1}^N \|\hat X_k - X_{i,k}\|_F^2= \frac{L_G^2}{N}\|\bold{X}_k-\bold{\hat X}_k\|_F^2.
				\end{aligned}
		\end{equation}
		
		Next, we turn to the $term\ 2$. Since $\bold{X}_k \in \mathcal{S}$ for $\forall k \geq 0$, we have $\|\bold{X}_{k} - \bold{\hat X}_{k}\|_F^2 \leq \frac{N}{2}$ by Definition \ref{l-r}. Then by using Lemma \ref{al-1}, we have
		\begin{equation}\label{3-4}
		\begin{aligned}
		&\|\hat X_{k+1}-\bar{X}_{k+1} + \hat X_{k}- \bar{X}_{k}\|_F^2\\
		\leq& 2\|\hat X_{k+1}-\bar{X}_{k+1}\|_F^2+2 \|\hat X_{k}- \bar{X}_{k}\|_F^2\\
		\leq& \frac{8r}{N^2}(\|\bold{X}_{k+1}-\hat{\bold{X}}_{k+1}\|_F^4+\|\bold{X}_{k}- \hat{\bold{X}}_{k}\|_F^4).
		\end{aligned}
		\end{equation}
		We also have $\frac{1}{N}\|\bold{X}_k-\bold{\hat X}_k\|_F^2 \leq \delta_1^2, $ from Definition \ref{l-r}. This together with
		 \eqref{3-4} yields $$
		term\ 2 \leq \frac{4r \delta_1^2}{N\beta^2L_G}\left(\|\bold{ X}_{k+1}-\hat{\bold{X}}_{k+1}\|_F^2+\|\bold{X}_{k}- \hat{\bold{X}}_{k}\|_F^2\right).$$
		
		For the $term\ 3$, since we have $\|\bar F_k\|_F \leq A$ according to Assumption \ref{assu-3} and Remark \ref{re-4}, it follows that
		\begin{align*}
		&term\ 3\\
		\leq & E[ \frac{1}{N} \| \bar F_k \|_F \cdot \|\sum_{i=1}^N ({X}_{i,k+1} - [{X}_{i,k}-\beta v_{i,k}-\alpha \text{grad}h_{i,t}(\bold{X}_k)]) \|_F |
		\mathcal{F}_k]\\
		\leq&\frac{A}{N}\sum_{i=1}^N E[\|{X}_{i,k+1} - [{X}_{i,k}-\beta v_{i,k}-\alpha \text{grad}h_{i,t}(\bold{X}_k)]\|_F| \mathcal{F}_k]_F.
		\end{align*}
		According to \eqref{re-u} and \eqref{Rx} in Lemma \ref{lem-1}, we have\begin{equation}\label{3-7}
		\begin{aligned}
		term\ 3\leq& \frac{AM}{N}\sum_{i=1}^N E[\|\beta v_{i,k}+\alpha \text{grad}h_{i,t}(\bold{X}_k)\|_F^2| \mathcal{F}_k]_F\\
		\leq& \frac{4AM\alpha^2}{N}\|\bold{X}_k-\bold{\hat X}_k\|_F^2+ \frac{2AM\beta^2}{N}E[\|\bold{Y}_k\|_F^2|\mathcal{F}_k],
		\end{aligned}
		\end{equation}
		where the last inequality follows by \eqref{9-2} in Lemma \ref{al-3} and $\|\bold{v}_k\|_F \leq \|\bold{Y}_k\|_F$.
		
		We now come to analyze the $term\ 4$. Since $N_{x_{i,k}\mathcal{M}}$ is defined as the orthogonal complement of $T_{x_{i,k}\mathcal{M}}$ with respect to the Euclidean space stated in Section \ref{sec-2.1}, we derive $Y_{i,k} - v_{i,k} = \mathcal{P}
		_{N_{x_{i,k}\mathcal{M}}}[Y_{i,k}]$. Together with $\text{grad} f(X_{i,k}) \in T_{x_{i,k}\mathcal{M}}$, we have \[\langle \text{grad} f(X_{i,k}) , \beta (Y_{i,k}-v_{i,k})\rangle = 0,\]which then gives\begin{equation}\label{3-10}
		\begin{aligned}
		term\ 4
		= & \frac{1}{N}\sum_{i=1}^N E[\langle \bar F_k - \text{grad} f(X_{i,k}) , \beta (Y_{i,k}-v_{i,k})\rangle |	\mathcal{F}_k] \\
		-& \frac{ \alpha}{N} E[\langle \bar F_k , \sum_{i=1}^N \text{grad}h_{i,t}(\bold{X}_k) \rangle| \mathcal{F}_k] .
		\end{aligned}
		\end{equation}
		
		For the first term of \eqref{3-10}, by Young's inequality, we have
		\begin{align}\label{3-12}
		&\frac{1}{N}\sum_{i=1}^N E[\langle \bar F_k - \text{grad} f(X_{i,k}) , \beta (Y_{i,k}-v_{i,k})\rangle |	\mathcal{F}_k] \notag\\
		\leq& \frac{1}{4NL_G}\sum_{i=1}^N E[\|\bar F_k - \text{grad} f(X_{i,k}) \|_F^2|	\mathcal{F}_k]\notag\\
		+&\frac{\beta^2 L_G}{N}\sum_{i=1}^N E[\|\mathcal{P}
		_{N_{x_{i,k}\mathcal{M}}}[Y_{i,k}]\|_F^2 | \mathcal{F}_k]\\
		\leq & \frac{1}{4NL_G}\sum_{i=1}^N E[\|\bar F_k - \text{grad} f(X_{i,k}) \|_F^2|	\mathcal{F}_k]+\frac{\beta^2 L_G}{N}E[\|\bold{Y}_k\|_F^2 | \mathcal{F}_k]\notag,
		\end{align}
		where the last inequality follows by the nonexpansiveness of the orthogonal projection.
		According to \eqref{oracle}, we have \begin{equation}\label{3-13}
		E[\|\bar F_k - \bar g_k\|_F^2|\mathcal{F}_k] \leq \frac{1}{N}\sum_{i=1}^N E[\|F_i(X_{i,k}) - \text{grad} f_i(X_{i,k})\|_F^2|\mathcal{F}_k]\leq \frac{\sigma^2}{ N_k},
		\end{equation}
		and by applying \eqref{lip} and Lemma \ref{al-4} we can derive \begin{equation*}
		\begin{aligned}
		&\frac{1}{N}\sum_{i=1}^N \|\text{grad} f(X_{i,k})- \bar g_k \|_F^2 \\
		\leq &2\| \text{grad} f(\hat X_k) - \bar g_k\|_F^2
		+\frac{2}{N}\sum_{i=1}^N \|\text{grad} f(X_{i,k})- \text{grad} f(\hat X_k) \|_F^2\\
		\leq &\frac{4L_G^2}{N}\|\bold{X}_k-\bold{\hat X}_k\|_F^2.
		\end{aligned}
		\end{equation*}
		This combining with \eqref{3-13} further derive \begin{equation*}
		\begin{aligned}
		&\frac{1}{4NL_G}\sum_{i=1}^N E[\|\bar F_k - \text{grad} f(X_{i,k}) \|_F^2|	\mathcal{F}_k]\\
		 \leq& \frac{1}{2L_G}E[\|\bar F_k - \bar g_k\|_F^2|\mathcal{F}_k]+\frac{1}{2L_GN}\sum_{i=1}^N \|\ \text{grad} f(X_{i,k})- \bar g_k \|_F^2\\
		\leq &\frac{\sigma^2}{2L_G N_k}+\frac{2L_G}{N}\|\bold{X}_k-\bold{\hat X}_k\|_F^2.
		\end{aligned}
		\end{equation*}
		Substituting the above bound into \eqref{3-12}, we have
		\begin{equation}\label{3-16}
		\begin{aligned}
		&\frac{1}{N}\sum_{i=1}^N E[\langle \bar F_k - \text{grad} f(X_{i,k}) , \beta (Y_{i,k}-v_{i,k})\rangle |	\mathcal{F}_k]\\
		 \leq &\frac{\sigma^2}{2L_G N_k}+\frac{2L_G}{N}\|\bold{X}_k-\bold{\hat X}_k\|_F^2+\frac{\beta^2 L_G}{N}E[\|\bold{Y}_k\|_F^2 | \mathcal{F}_k].
		\end{aligned}
		\end{equation}
		For the second term of \eqref{3-10}, by using \eqref{9-2} in Lemma \ref{al-3} and the parameter $L_t \leq 2$, we can derive
		\begin{equation*}
		\begin{aligned}
		\frac{ \alpha}{N} E[\langle \bar F_k , \sum_{i=1}^N \text{grad}h_{i,t}(\bold{X}_k) \rangle| \mathcal{F}_k] \leq &\frac{2\alpha}{N}E[\|\bar F_k\|_F\cdot \|\bold{X}_k-\bold{\hat X}_k\|_F^2|\mathcal{F}_k] \\
		\leq& \frac{2\alpha A}{N} \|\bold{X}_k-\bold{\hat X}_k\|_F^2,
		\end{aligned}
		\end{equation*}
		where the last inequality follows by $\|\bar F_k\|_F \leq A$.
		This together with  \eqref{3-16} gives the estimation of $term\ 4$ as
		\begin{equation}
		term\ 4 \leq \frac{\sigma^2}{2L_G N_k}+\frac{2L_G-2\alpha A}{N}\|\bold{X}_k-\bold{\hat X}_k\|_F^2+\frac{\beta^2 L_G}{N}E[\|\bold{Y}_k\|_F^2 | \mathcal{F}_k].
		\end{equation}
		
		Finally, we estimate the $term \ 5$. We have the following analysis about the average error
		\begin{align}\label{3-17}
		&\|\bar X_{k+1}-\bar X_k\|_F= \|\frac{1}{N}\sum_{i=1}^N(X_{i,k+1}-X_{i,k}) \|_F \notag\\
		=&\|\frac{1}{N}\sum_{i=1}^N(X_{i,k+1}-X_{i,k}+\beta v_{i,k}+\alpha \text{grad}h_{i,t}(\bold{X}_k))\notag\\
		-&\frac{1}{N}\sum_{i=1}^N(\beta v_{i,k}+\alpha \text{grad}h_{i,t}(\bold{X}_k))\|_F\\
		\leq &\frac{1}{N}\sum_{i=1}^N\|X_{i,k+1}-(X_{i,k}-\beta v_{i,k}-\alpha \text{grad}h_{i,t}(\bold{X}_k))\|_F\notag\\
		+&\beta\|\bar v_{k}\|_F+ \frac{\alpha}{N}\| \sum_{i=1}^N \text{grad}h_{i,t}(\bold{X}_k)\|_F.\notag
		\end{align}
	    
		By using \eqref{re-u}, and \eqref{Rx} in Lemma \ref{lem-1},
		\[\begin{aligned}
		\|\bar X_{k+1}-\bar X_k\|_F\leq& \frac{M}{N}\sum_{i=1}^N\|\beta v_{i,k}+\alpha \text{grad}h_{i,t}(\bold{X}_k))\|_F^2+\beta\|\bar v_{k}\|_F\\
		+&\frac{\alpha}{N}\|\sum_{i=1}^N \text{grad}h_{i,t}(\bold{X}_k) \|_F\\
		\leq & \frac{2M\beta^2}{N}\|\bold{v}_k\|_F^2+\frac{2M\alpha^2}{N}\|\text{grad}h_t(\bold{X}_k)\|_F^2\\
		+&\beta\|\bar v_{k}\|_F+\frac{\alpha}{N} \| \sum_{i=1}^N \text{grad}h_{i,t}(\bold{X}_k) \|_F.
		\end{aligned}\]
		Applying \eqref{9-1} and \eqref{9-2} in Lemma \ref{al-3}, we get\[\|\bar X_{k+1}-\bar X_k\|_F\leq \frac{2M\beta^2}{N}\|\bold{v}_k\|_F^2+\frac{2M\alpha^2L_t^2+\alpha L_t}{N}\|\bold{X}_k-\bold{\hat X}_k\|_F^2+\beta\|\bar v_{k}\|_F.\]
		Since $\alpha \leq \frac{1}{M},\ L_t \leq 2, \|\bold{ v}_k\|_F \leq \|\bold{Y}_k\|_F$ hold, we then derive \begin{equation}\label{3-18}
		\|\bar X_{k+1}-\bar X_k\|_F\leq \frac{2M\beta^2}{N}\|\bold{Y}_k\|_F^2+ \frac{10 \alpha}{N}\|\bold{X}_k-\bold{\hat X}_k\|_F^2 +\beta\|\bar v_{k}\|_F.
		\end{equation}
		For the term $\|\bar v_{k}\|_F$, we have the following estimation. By Lemma \ref{gt}, we have \begin{equation*}
		\begin{aligned}
		\|\bar v_k\|_F^2 \leq &2\|\bar F_k\|_F^2+2\|\bar v_k - \bar Y_k\|_F^2\leq 2\|\bar F_k\|_F^2+\frac{2}{N}\sum_{i=1}^N\|Y_{i,k} - v_{i,k}\|_F^2\\
		\leq& 2\|\bar F_k\|_F^2+\frac{2}{N}\|\bold{Y}_k\|_F^2,
		\end{aligned}
		\end{equation*}where the last inequality follows by the non-expansiveness of the orthogonal projection since $Y_{i,k} - v_{i,k} = \mathcal{P}
		_{N_{x_{i,k}\mathcal{M}}}[Y_{i,k}]$.
		This combined with \eqref{3-18} implies \begin{align*}
		E[\|\bar X_{k+1}-\bar X_k\|_F^2|\mathcal{F}_k]	\leq &2E[( \frac{10 \alpha}{N}\|\bold{X}_k-\bold{\hat X}_k\|_F^2+\frac{2M\beta^2}{N}\|\bold{Y}_k\|_F^2)^2|\mathcal{F}_k]\\
		+&4\beta^2 E[\|\bar F_k\|_F^2+\frac{1}{N}\|\bold{Y}_k\|_F^2|\mathcal{F}_k].
		\end{align*}
		
		\noindent Since $E[\langle \bar F_k - \bar g_k, \bar g_k \rangle|\mathcal{F}_k]=0$ according to Assumption \ref{assu-2}, we have \begin{equation}\label{3-19}
		E[\|\bar F_k\|_F^2 |\mathcal{F}_k] = E[\|\bar F_k - \bar g_k\|_F^2 |\mathcal{F}_k]+\|\bar g_k\|_F^2\leq \frac{\sigma^2}{N_k}+ \|\bar g_k\|_F^2,
		\end{equation} where the inequality follows by the independence of $F_i(X_{i,k})$ and \eqref{oracle}, which further implies that \begin{equation}\label{3-21}
		\begin{aligned}
		&E[\|\bar X_{k+1}-\bar X_k\|_F^2|\mathcal{F}_k]\\
		\leq &4\left( \frac{100 \alpha^2}{N^2}\|\bold{X}_k-\bold{\hat X}_k\|_F^4+\frac{4M^2\beta^4}{N^2}E[\|\bold{Y}_k\|_F^4|\mathcal{F}_k]\right)\\
		+&4\beta^2 \left(\frac{1}{N}E[\|\bold{Y}_k\|_F^2|\mathcal{F}_k]\right)+4\beta^2\left(\frac{\sigma^2}{N_k}+ \|\bar g_k\|_F^2\right).
		\end{aligned}
		\end{equation}
		By Lemma \ref{lem-3}, we have $\beta \|\bold{Y}_k\|_F \leq \frac{\alpha \delta_1}{5}\sqrt{N}$ . This together with the bound of local region $\mathcal{S}_1$ produces
		\begin{align*}
		&E[\|\bar X_{k+1}-\bar X_k\|_F^2|\mathcal{F}_k] \\
		\leq & 4\left( \frac{100\delta_1^2 \alpha^2}{N}\|\bold{X}_k-\bold{\hat X}_k\|_F^2+\frac{(M \alpha \beta\delta_1)^2}{5N}E[\|\bold{Y}_k\|_F^2|\mathcal{F}_k]\right)\\
		+&4 \beta^2 \left(\frac{1}{N}E[\|\bold{Y}_k\|_F^2|\mathcal{F}_k]\right)+4\beta^2\left(\frac{\sigma^2}{N_k}+ \|\bar g_k\|_F^2\right).
		\end{align*}
		For simplicity, let $\bar{\sigma}_{r,k}:=\sigma_r(\bar{X}_{k})$. Since \eqref{col-1} implies that $\left(\frac{2}{(\bar \sigma_{r,k+1}+\bar \sigma_{r,k})}\right)^2 \leq \frac{1}{(1-2\delta_1^2)^2} $, then by using \eqref{al-2} and \eqref{3-21}, we have $$\begin{aligned}
		&term\ 5 = E[\|\bold{\hat X}_{k+1} - \bold{\hat X}_{k}\|_F^2|\mathcal{F}_k] \leq  \frac{1}{(1-2\delta_1^2)^2}E[\|\bar X_{k+1}-\bar X_k\|_F^2|\mathcal{F}_k]\\
		\leq &\frac{4}{(1-2\delta_1^2)^2}\left( \frac{100\delta_1^2 \alpha^2}{N}\|\bold{X}_k-\bold{\hat X}_k\|_F^2+\frac{(M \alpha \beta\delta_1)^2}{5N}E[\|\bold{Y}_k\|_F^2|\mathcal{F}_k]\right)\\
		+&\frac{4\beta^2}{(1-2\delta_1^2)^2}\left(\frac{1}{N}E[\|\bold{Y}_k\|_F^2|\mathcal{F}_k]\right)+\frac{4\beta^2}{(1-2\delta_1^2)^2}\left(\frac{\sigma^2}{N_k}+ \|\bar g_k\|_F^2\right).
		\end{aligned}$$
		
		Since $\alpha \leq \frac{1}{M}$ and $\delta_1 \leq \frac{1}{5\sqrt{r}}\delta_2\leq \frac{1}{30\sqrt{r}}$ by Definition \ref{l-r}, we can estimate that $\frac{300\delta_1^2 \alpha^2}{N(1-2\delta_1^2)^2}\leq \frac{13\alpha^2\delta_2^2}{Nr}$, $\frac{3\beta^2}{(1-2\delta_1^2)^2} \leq 4\beta^2$, and $\frac{3(M \alpha \beta\delta_1)^2}{5N(1-2\delta_1^2)^2} \leq \frac{\delta_1^2 \beta^2}{N}$.
		We then derive that
		\begin{align*}
		\frac{3L_G}{4}term\ 5\leq &\frac{13\alpha^2L_G\delta_2^2}{Nr}\|\bold{X}_k-\bold{\hat X}_k\|_F^2+4L_G \beta^2\|\bar g_k\|_F^2\\
		+&\frac{(\delta_1^2+4)L_G\beta^2}{N}E[\|\bold{Y}_k\|_F^2|\mathcal{F}_k]+\frac{4 L_G\beta^2 \sigma^2}{N_k}.
		\end{align*}
		Finally, we will combine all the terms. According to the upper bounds of $term\ 1,2,3,4,5$, we derive that
		\begin{align*}
		&E[f(\hat X_{k+1})|\mathcal{F}_k] \leq f(\hat X_k)-(\beta - \frac{\beta^2 L_G}{2})\|\bar g_k\|_F^2\\
		+&\frac{1}{L_G}term\ 1+term\ 2+term\ 3+term\ 4+\frac{3L_G}{4}term\ 5\\
		\leq & f(\hat X_k)-(\beta-\frac{9}{2}L_G\beta^2)\|\bar g_k\|_F^2+\frac{(1+4 L_G^2 \beta^2)\sigma^2}{2L_G N_k}\\
		+&\frac{3L_G+2\alpha A+13L_G\tilde C \alpha^2+\frac{4r \delta_1^2}{\beta^2L_G}}{N}\|\bold{X}_k-\bold{\hat X}_k\|_F^2\\
		+&\frac{(\delta_1^2+5)L_G+2MA}{N}\beta^2E[\|\bold{Y}_k\|_F^2|\mathcal{F}_k]+\frac{4r \delta_1^2}{N\beta^2L_G}\|\bold{ X}_{k+1}-\hat{\bold{X}}_{k+1}\|_F^2,
		\end{align*}
		where $\tilde C = \frac{\delta_2^2}{r }$. Since $\beta - \frac{9\beta^2 L_G}{2}\geq \tfrac{\beta}{2}$ by  $\beta \leq \frac{1}{9L_G}$, the lemma is proved with the definitions of $D_1, D_2, D_3, D_4$.
	\end{proof}
	
\subsection{Main Results}
	Next, we take $E[\|\text{grad}f(\hat X)\|_F^2]$ as the metric and introduce the convergence theorem and convergence rate of Algorithm \ref{alg:1}.
	\begin{theorem}\label{th-1}
		 Let  Assumption \ref{assu-1}-\ref{assu-3} hold. Consider Algorithm \ref{alg:1},  where  $\bold{X}_0 \in \mathcal{S}$ and  $t \geq [\log_{\sigma_2}(\frac{1}{2\sqrt{N}})]$. Suppose, in addition, that the positive step size $\alpha$ satisfies $\alpha \leq \min\{1,\frac{1}{M},\frac{\Phi}{L_t}\}$ with $\Phi $ and $L_t$ defined in Lemma \ref{al-6} and Lemma \ref{al-3}, and step size $\beta$ satisfies $$\beta \leq \min\{\bar \beta,\frac{1}{9L_G}, \frac{1}{4[\alpha \delta_1L_G(16 \Gamma_0+\Gamma_2)+2C_1]}\},$$where $\bar \beta$ is given in Lemma \ref{lem-3}. Denote by $\hat X_{k+1}$ the IAM in \eqref{iam} and by $\bold{\hat X}: = (\bm{1}_N \otimes I_n) \hat X$. Then we have
		\begin{align}\label{18}
		&\min_{k=0,\dots, K}\frac{1}{N}E[\|\bold{X}_k-\bold{\hat X}_k\|_F^2]\\
		\leq& \frac{8\beta \Gamma_0\left( f(\hat X_0)-f^*+\tilde{\Gamma}_3+\sigma^2\sum_{k=0}^K(\frac{2\Gamma_2\beta+ D_3+\frac{\beta}{2}}{N_k}+\frac{2\beta\Gamma_2}{N_{k+1}})\right)+\Gamma_1}{K+1},\notag
		\end{align}
		\begin{align}\label{19}
			\min_{k=0,\dots, K}E[\|\text{grad}f(\hat X_k)\|_F^2]\leq &\frac{\tilde \Gamma_1+(16+\Gamma_0 \alpha^2\delta_1^2)[f(\hat X_0)-f^*+\tilde{\Gamma}_3]}{\beta \cdot (K+1)} \notag\\
			+&\frac{\sigma^2\sum_{k=0}^K\left(\frac{2\beta\Gamma_2+\tilde D_3}{N_k}+\frac{2\beta\Gamma_2}{N_{k+1}}\right)}{\beta \cdot (K+1)},
		\end{align}
		where $\Gamma_0 $ - $ \Gamma_3$ are defined in Lemma \ref{lem-5}, $ \tilde \Gamma_1 = 2L_G^2 \beta \Gamma_1, \tilde{\Gamma}_3 = C_0+\frac{\beta}{2}(\Gamma_3+16\alpha^2 L_G^2 \Gamma_2 \Gamma_1)$; $ C_0 = \frac{D_4 \Gamma_0\alpha^2 \delta_1^2}{25}+\Gamma_1(D_2+D_4)$, $C_1 = D_2\Gamma_0+D_4\Gamma_0+D_1$,  and $\tilde D_3: = D_3+\frac{5\beta}{8}$ with $D_1$ - $D_4$ defined in Lemma \ref{lem-4}.
	\end{theorem}
	\begin{proof}
		By taking unconditional expectations of the sufficiently descent inequality derived in Lemma \ref{lem-4}, we obtain that
		\begin{equation}\label{rec}
			\begin{aligned}
			&E[f(\hat X_{k+1})] \\
			\leq & E[f(\hat X_k)]-\frac{\beta}{2}E[\|\bar g_k\|_F^2]+\frac{D_1}{N}\beta^2E[\|\bold{Y}_k\|_F^2]\\
			+&\frac{D_2}{N}E[\|\bold{X}_k-\bold{\hat X}_k\|_F^2]+\frac{\sigma^2D_3}{ N_k}+\frac{D_4}{N}E[\|\bold{ X}_{k+1}-\hat{\bold{X}}_{k+1}\|_F^2].
			\end{aligned}
		\end{equation}
		Then, by taking unconditional expectations of \eqref{l-4}, we have
		\begin{equation}\label{t-1}
		\frac{1}{N}\sum_{k=0}^KE[\|\bold{X}_k-\bold{\hat X}_k\|_F^2] \leq \Gamma_0\beta^2 \frac{1}{N}\sum_{k=0}^KE[\|\bold{Y}_k\|_F^2] +\Gamma_1.
		\end{equation}
		By taking summation of \eqref{rec} from $k=0$ to $K$ and substituting \eqref{t-1} into it, we have\begin{align*}
		&E[f(\hat X_{k+1})]\\
		\leq & f(\hat X_0)- \frac{\beta}{2}\sum_{k=0}^K E[\|\bar g_k\|_F^2]+\frac{D_1}{N}\beta^2\sum_{k=0}^K E[\|\bold{Y}_k\|_F^2]+D_3\sigma^2\sum_{k=0}^K \frac{1}{N_k}\\
		+&\frac{D_2}{N}\sum_{k=0}^KE[\|\bold{X}_k-\bold{\hat X}_k\|_F^2]+\frac{D_4}{N}\sum_{k=1}^{K+1}E[\|\bold{ X}_{k+1}-\hat{\bold{X}}_{k+1}\|_F^2]\\
		\leq & f(\hat X_0)-\frac{\beta}{2}\sum_{k=0}^K E[\|\bar g_k\|_F^2]+\frac{D_2\Gamma_0+D_4\Gamma_0+D_1}{N}\beta^2\sum_{k=0}^K E[\|\bold{Y}_k\|_F^2]\\
		+&D_3\sigma^2\sum_{k=0}^K \frac{1}{N_k}+\frac{D_4 \Gamma_0 \beta^2E[\|\bold{Y}_{k+1}\|_F^2]}{N}+\Gamma_1(D_2+D_4).
		\end{align*}
		Since $\beta \leq \frac{\alpha \delta_1}{5(10A+L_G)}$ and $\|\bold{Y}_{k+1}\|_F$ is bounded according to Lemma \ref{lem-3}, we have $\beta^2\|\bold{Y}_{k+1}\|_F^2 \leq \frac{\alpha^2 \delta_1^2 N}{25}$, which further implies
		\begin{equation}\label{t-4}
			\begin{aligned}
			E[f(\hat X_{k+1})]\leq & f(\hat X_0)-\frac{\beta}{2}\sum_{k=0}^K E[\|\bar g_k\|_F^2]+D_3\sigma^2\sum_{k=0}^K \frac{1}{N_k}+C_0\\
			+&\frac{C_1}{N}\beta^2\sum_{k=0}^K E[\|\bold{Y}_k\|_F^2],
			\end{aligned}
		\end{equation} where $C_0 = \frac{D_4 \Gamma_0\alpha^2 \delta_1^2}{25}+\Gamma_1(D_2+D_4), C_1 = D_2\Gamma_0+D_4\Gamma_0+D_1$.
		
		Next, we analyze the relation between $E[\|\bar g_k\|_F^2]$ and $E[\|\bold{Y}_k\|_F^2]$. By taking expectation of \eqref{3-19}, we have\begin{equation}\label{t-9}
			-E[\|\bar g_k\|_F^2] \leq \frac{\sigma^2}{N_k}-E[\|\bar F_k\|_F^2].
		\end{equation}Since  $\|\bold{Y}_k\|_F^2 \leq 2\|\bold{Y}_k -\bar{\bold{F}}_k\|_F^2+2\|\bar{\bold{F}}_k\|_F^2$, we can derive
		$$-\sum_{k=0}^K\|\bar F_k\|_F^2 = -\frac{1}{N}\sum_{k=0}^K\|\bar{\bold{F}}_k\|_F^2 \leq \frac{1}{N}\sum_{k=0}^K\|\bold{Y}_k-\bar{\bold{F}}_k\|_F^2-\frac{1}{2N}\sum_{k=0}^K\|\bold{Y}_k\|_F^2.$$
		By taking expectation of the above inequality and together with \eqref{t-9}, we have \begin{equation}\label{t-2}
		\begin{aligned}
		-\sum_{k=0}^KE[\|\bar g_k\|_F^2] \leq& \frac{1}{N}\sum_{k=0}^KE[\|\bold{Y}_k-\bar{\bold{F}}_k\|_F^2]\\
		-&\frac{1}{2N}\sum_{k=0}^KE[\|\bold{Y}_k\|_F^2]+\sum_{k=0}^K \frac{\sigma^2}{N_k}.
		\end{aligned}
		\end{equation}
		
		We proceed to give an upper bound on the first term of the right-hand side of \eqref{t-2}. Note that
		\begin{equation}\label{1-1}
		\begin{aligned}
        \|F(\bold{X}_{k+1})-F(\bold{X}_k)\|_F \leq &\|g (\bold{X}_{k+1}) - g (\bold{X}_{k})\|_F+ \|F(\bold{X}_k) - g(\bold{X}_k)\|_F\\
        +&\|F(\bold{X}_{k+1}) - g(\bold{X}_{k+1})\|_F.
		\end{aligned}
		\end{equation}
		By Lemma \ref{al-4}, we have \begin{equation}\label{1-2}
		\|g (\bold{X}_{k+1}) - g(\bold{X}_{k})\|_F \leq L_G\|\bold{X}_{k+1}-\bold{X}_k\|_F.
		\end{equation}
		Plugging in the update \eqref{re-u}, and by Lemma \ref{lem-1} and \eqref{9-2} in Lemma \ref{al-3}, we derive that $\|\bold{X}_{k+1}-\bold{X}_k\|_F \leq \alpha\|\text{grad} h_t(\bold{X}_k)\|_F+\beta \|\bold{v}_k\|_F \leq  2\alpha\|\bold{X}_k-\bold{\hat X}_k\|_F+\beta \|\bold{v}_k\|_F.$
		Since $\|\bold{v}_k\|_F \leq \|\bold{Y}_k\|_F$, we further have \begin{equation}\label{1-3}
		\|\bold{X}_{k+1}-\bold{X}_k\|_F \leq  2\alpha\|\bold{X}_k-\bold{\hat X}_k\|_F+\beta \|\bold{Y}_k\|_F.
		\end{equation}
		Combining \eqref{1-1}-\eqref{1-3}, we get \begin{equation*}
		\begin{aligned}
		\|F(\bold{X}_{k+1})-F(\bold{X}_k)\|_F \leq &2\alpha L_G\|\bold{X}_k-\bold{\hat X}_k\|_F+\beta L_G \|\bold{Y}_k\|_F\\
		+& \|F(\bold{X}_k) - g(\bold{X}_k)\|_F\\
		+&\|F(\bold{X}_{k+1}) - g(\bold{X}_{k+1})\|_F.
		\end{aligned}
		\end{equation*}
		
		Then, plugging the above result into the accumulated gradient tracking error \eqref{l-2} in Lemma \ref{lem-5} yields \begin{align*}
		\frac{1}{N}\sum_{k=0}^KE[\|\bold{Y}_k-\bar{\bold{F}}_k\|_F^2]\leq &\Gamma_2\frac{1}{N}\sum_{k=0}^K\Big(16\alpha^2 L_G^2E[\|\bold{X}_k-\bold{\hat X}_k\|_F^2]\\
		+&4\beta^2L_G^2E[\|\bold{Y}_k\|_F^2]\\
		+&4E[\|F(\bold{X}_{k+1})-g(\bold{X}_{k+1})\|_F^2]\\
		+&4E[\|F(\bold{X}_{k})-g(\bold{X}_{k})\|_F^2]\Big)+\Gamma_3,
		\end{align*}where the inequality also uses $\|a+b+c+d\|^2 \leq 4(\|a\|^2+\|b\|^2+\|c\|^2+\|d\|^2)$.
		Due to the bounded variance of noise characterized by \eqref{v-s} from Assumption \ref{assu-2}, it follows that\begin{equation*}
		\begin{aligned}
		\frac{1}{N}\sum_{k=0}^KE[\|\bold{Y}_k-\bar{\bold{F}}_k\|_F^2]\leq& \Gamma_3+\Gamma_2\frac{1}{N}\sum_{k=0}^K\Big(16\alpha^2 L_G^2E[\|\bold{X}_k-\bold{\hat X}_k\|_F^2]\\
		+&4\beta^2L_G^2E[\|\bold{Y}_k\|_F^2]+\frac{4N\sigma^2}{N_{k+1}}+\frac{4N\sigma^2}{N_{k}}\Big).
		\end{aligned}
		\end{equation*}
		This combined with \eqref{t-1} implies
		$$\begin{aligned}
		&\frac{1}{N}\sum_{k=0}^KE[\|\bold{Y}_k-\bar{\bold{F}}_k\|_F^2]\leq \frac{1}{N}\sum_{k=0}^K(16\alpha^2  \Gamma_2 \Gamma_0+4\Gamma_2)\beta^2 L_G^2E[\|\bold{Y}_k\|_F^2]\\
		+&4\Gamma_2\sigma^2 \sum_{k=0}^K (\frac{1}{N_{k+1}}+\frac{1}{N_{k}})
		+16\alpha^2 L_G^2 \Gamma_2 \Gamma_1+\Gamma_3.
		\end{aligned}$$
		By $t \geq [\log_{\sigma_2}(\frac{1}{2\sqrt{N}})]$, we get $\alpha^2 \Gamma_2 \leq \Gamma_2\leq \frac{2}{(1-\frac{1}{2 \sqrt{N}})^2} \leq 5$ for $\Gamma_2 = \frac{2}{(1-\sigma_2^t)^2}$ stated in Lemma \ref{lem-5}. This together with $\beta \leq \frac{\alpha \delta_1}{5(10A+L_G)}$ implies $(16\alpha^2  \Gamma_2 \Gamma_0+4\Gamma_2)\beta L_G \leq (16 \Gamma_0+\Gamma_2)\alpha \delta_1$, hence
		\begin{align*}
	    &\frac{1}{N}\sum_{k=0}^KE[\|\bold{Y}_k-\bar{\bold{F}}_k\|_F^2]\leq\frac{1}{N}\sum_{k=0}^K(16 \Gamma_0+\Gamma_2)\alpha \delta_1 \beta L_GE[\|\bold{Y}_k\|_F^2]\notag\\
	    +&4\Gamma_2\sigma^2 \sum_{k=0}^K (\frac{1}{N_{k+1}}+\frac{1}{N_{k}})+16\alpha^2 L_G^2 \Gamma_2 \Gamma_1+\Gamma_3.
		\end{align*}
		Substituting this into  \eqref{t-2} yields $$\begin{aligned}
		-\sum_{k=0}^KE[\|\bar g_k\|_F^2]  \leq & \Big[(16 \Gamma_0+\Gamma_2)\alpha \delta_1 \beta L_G-\frac{1}{2}]\frac{1}{N}\sum_{k=0}^KE[\|\bold{Y}_k\|_F^2\Big]\\
		+&\sigma^2 \sum_{k=0}^K (\frac{4\Gamma_2}{N_{k+1}}+\frac{4\Gamma_2+1}{N_{k}})+\Gamma_3+16\alpha^2 L_G^2 \Gamma_2 \Gamma_1.
		\end{aligned}$$
		
		Therefore, plugging the above result into \eqref{t-4} implies
		\begin{align*}
		&E[f(\hat X_{k+1})]\\
		\leq & f(\hat X_0)- \frac{\beta}{2N}\left[\frac{1}{2}-\beta(\alpha \delta_1L_G(16 \Gamma_0+\Gamma_2)+2C_1)\right] \sum_{k=0}^KE[\|\bold{Y}_k\|_F^2]\\
		+&\frac{\beta}{2}(\Gamma_3+16\alpha^2 L_G^2 \Gamma_2 \Gamma_1)+C_0\\
		+&D_3\sigma^2\sum_{k=0}^K \frac{1}{N_k}+ \sigma^2\beta \sum_{k=0}^K (\frac{2 \Gamma_2}{N_{k+1}}+\frac{2 \Gamma_2+\frac{1}{2}}{N_{k}}).
		\end{align*}
		Since $\beta \leq \frac{1}{4[\alpha \delta_1L_G(16 \Gamma_0+\Gamma_2)+2C_1]}$, it follows that \begin{align*}
		E[f(\hat X_{k+1})]
		\leq &f(\hat X_0)-\frac{\beta}{8}\frac{1}{N}\sum_{k=0}^KE[\|\bold{Y}_k\|_F^2]\\
		+&\sigma^2\sum_{k=0}^K\left(\frac{2\beta \Gamma_2+ D_3+\frac{\beta}{2}}{N_k}+\frac{2\beta \Gamma_2}{N_{k+1}}\right)+\tilde{\Gamma}_3,
		\end{align*}where $ \tilde{\Gamma}_3 := \frac{\beta}{2}(\Gamma_3+16\alpha^2 L_G^2 \Gamma_2 \Gamma_1) + C_0$. We then derive the upper bound of $E[\|\bold{Y}_k\|_F^2]$ with $f^* = \min_{X \in \text{St}(n,r)} f(X)$:
		\begin{equation}\label{y_k}
			\begin{aligned}
			&\frac{\beta}{8}\frac{1}{N}\sum_{k=0}^KE[\|\bold{Y}_k\|_F^2] \\
			\leq& f(\hat X_0)-f^*+\sigma^2\sum_{k=0}^K\left(\frac{2\beta \Gamma_2+D_3+\frac{\beta}{2}}{N_k}+\frac{2\beta\Gamma_2}{N_{k+1}}\right)+\tilde{\Gamma}_3.
			\end{aligned}
		\end{equation}
		By Assumption \ref{assu-2} and Lemma \ref{gt}, we can get $E[\|\bar g_k\|_F^2] = E[\|\bar g_k - \bar F_k\|_F^2]+E[\|\bar Y_k\|_F^2] \leq \frac{\sigma^2}{N_k} + E[\frac{1}{N}\sum_{i=1}^N\|Y_{i,k}\|_F^2]
		 = \frac{\sigma^2}{N_k}+\frac{1}{N}E[\|{\bold{Y}}_k\|_F^2]$ which then implies \begin{align*}
		&\frac{\beta}{8}\sum_{k=0}^KE[\|\bar g_k\|_F^2] \leq \frac{\beta}{8}\sum_{k=0}^K \frac{ \sigma^2}{N_k}+ \frac{\beta}{8}\frac{1}{N}\sum_{k=0}^KE[\|\bold{Y}_k\|_F^2] \\
		\leq& f(\hat X_0)-f^*+\sigma^2\sum_{k=0}^K\left(\frac{2\beta \Gamma_2+D_3+\frac{5\beta}{8}}{N_k}+\frac{2\beta \Gamma_2}{N_{k+1}}\right)+\tilde{\Gamma}_3.
		\end{align*}
		Denote $\tilde D_3: = D_3+\frac{5\beta}{8}$. By using \eqref{y_k} we further have  \begin{equation}\label{t-5}
		\begin{aligned}
		&\min_{k=0,\dots,K}E[\|\bar g_k\|_F^2] \leq \frac{1}{K+1}\sum_{k=0}^KE[\|\bar g_k\|_F^2]\\
		\leq &\frac{8\left( f(\hat X_0)-f^*+\tilde{\Gamma}_3+\sigma^2\sum_{k=0}^K(\frac{2\beta \Gamma_2+\tilde D_3}{N_k}+\frac{2\beta \Gamma_2}{N_{k+1}})\right)}{\beta \cdot (K+1)}.
		\end{aligned}
		\end{equation}
		
		Since $\min_{k=0,\dots, K}\frac{1}{N}E[\|\bold{X}_k-\bold{\hat X}_k\|_F^2]\leq \frac{1}{K+1}\sum_{k=0}^K\frac{1}{N}E[\|\bold{X}_k-\bold{\hat X}_k\|_F^2]$, we derive \eqref{18} by plugging \eqref{y_k} into \eqref{t-1}.
		By applying \eqref{lip}, we have
		\begin{align}\label{t-7}
		E\|\text{grad}f(\hat X_k)\|_F^2\leq &2\|\bar g_k\|_F^2+2E[\|\text{grad}f(\hat X_k) - \bar g_k\|_F^2] \notag\\
		\leq& 2\|\bar g_k\|_F^2+\frac{2L_G^2}{N}E[\|\bold{X}_k-\bold{\hat X}_k\|_F^2].
		\end{align}

		Since $\beta \leq \frac{\alpha \delta_1}{5(10A+L_G)}\leq \frac{\alpha \delta_1}{5L_G}$, we have $16 \Gamma_0 \beta^2 L_G^2 \leq \Gamma_0 \alpha^2 \delta_1^2$. Then by combining  \eqref{18}, \eqref{t-5}, and \eqref{t-7}, we finally derive \eqref{19}.
	\end{proof}
	
    The following theorem analyzes the impact of sample size $N_k$ on the convergence rate.
	\begin{theorem}\label{th-2}
		 Let  Assumption \ref{assu-1}-\ref{assu-3} hold. Consider Algorithm \ref{alg:1},  where  $\bold{X}_0 \in \mathcal{S}$ and  $t \geq [\log_{\sigma_2}(\frac{1}{2\sqrt{N}})]$. Suppose, in addition, that the step size $\alpha$ satisfies $\alpha \leq \min\{1,\frac{1}{M},\frac{\Phi}{L_t}\}$ with $\Phi $ and $L_t$ defined in Lemma \ref{al-6} and Lemma \ref{al-3}, and step size $\beta$ satisfies $$\beta \leq \min\{\bar \beta,\frac{1}{9L_G}, \frac{1}{4[\alpha \delta_1L_G(16 \Gamma_0+\Gamma_2)+2C_1]}\},$$where $\bar \beta$ is given in Lemma \ref{lem-3}. Then, we have
		
		\noindent (i) When $N_k = [q^{-k}]$, where $q \in(0,1)$, it holds
			\begin{align}\label{20}
		    &\min_{k=0,\dots, K}E[\|\text{grad}f(\hat X_k)\|_F^2] \\
		    \leq& \frac{(16+\Gamma_0 \alpha^2\delta_1^2)\left[f(\hat X_0)-f^*+\tilde{\Gamma}_3+\frac{\sigma^2(2\beta\Gamma_2(1+q)+\tilde D_3)}{1-q}\right]+\tilde \Gamma_1}{\beta \cdot  (K+1)}; \notag
			\end{align}
			(ii) When $N_k = [(k+1)^a](a>0)$, for $a=1$, it holds
			\begin{align}\label{21-1}
			\min_{k=0,\dots, K}E[\|\text{grad}f&(\hat X_k)\|_F^2] \leq\frac{\tilde \Gamma_1+(16+\Gamma_0 \alpha^2\delta_1^2)[ f(\hat X_0)-f^*+\tilde{\Gamma}_3}{\beta \cdot  (K+1)}\notag\\
			&+\frac{2\sigma^2(4\beta\Gamma_2+\tilde D_3)(\ln(K+1)+1)]}{\beta \cdot  (K+1)};
			\end{align}
			and for $a\neq 1$, it holds
			\begin{align}\label{21-2}
			&\min_{k=0,\dots, K}E[\|\text{grad}f(\hat X_k)\|_F^2]\leq\frac{\sigma^2(16+\Gamma_0 \alpha^2\delta_1^2)(6 \Gamma_2+\tilde D_3)}{\beta(1-a)(K+1)^a} \notag \\
			+&\frac{(16+\Gamma_0 \alpha^2\delta_1^2)[f(\hat X_0)-f^*+\tilde{\Gamma}_3]+\tilde \Gamma_1}{\beta \cdot  (K+1)},\text{ if }a<1; \notag \\
			&\min_{k=0,\dots, K}E[\|\text{grad}f(\hat X_k)\|_F^2] \leq \frac{\sigma^2(16+\Gamma_0 \alpha^2\delta_1^2)(3\Gamma_2a+\Gamma_2+\frac{a \tilde D_3}{\beta})}{(a-1)(K+1)} \notag\\
			+&\frac{(16+\Gamma_0 \alpha^2\delta_1^2)[f(\hat X_0)-f^*+\tilde{\Gamma}_3]+\tilde \Gamma_1}{\beta \cdot  (K+1)}, \text{ if }a>1;
			\end{align}
			(iii) When $N_k = Q$, where $Q >0$ is a constant, it holds
			\begin{align}\label{22}
			&\min_{k=0,\dots, K}E[\|\text{grad}f(\hat X_k)\|_F^2]\leq(16+\Gamma_0 \alpha^2\delta_1^2)\frac{\sigma^2(4\beta\Gamma_2+\tilde D_3)}{\beta Q} \notag \\
			+&\frac{\tilde \Gamma_1+(16+\Gamma_0 \alpha^2\delta_1^2)[f(\hat X_0)-f^*+\tilde{\Gamma}_1]}{\beta \cdot  (K+1)}.
			\end{align}
	\end{theorem}
	
	\begin{proof}
			(i) If  $N_k= [q^{-k}]$, then $\sum_{k=0}^K \frac{1}{N_k} \leq  \sum_{k=0}^K q^k = \frac{1-q^{K+1}}{1-q}\leq \frac{1}{1-q}$.
			By the same token, it holds $\sum_{k=0}^K \frac{1}{N_{k+1}} \leq  \sum_{k=0}^K q^{k+1} \leq  \frac{q}{1-q}.$
			
			Plugging the above results into \eqref{19} we derive \eqref{20}.
			
			(ii)  If $N_k = [(k+1)^a] (a>0)$, then $\sum_{k=0}^K \frac{1}{N_k} \leq \sum_{k=0}^K \frac{1}{(k+1)^a}$.
			Denote a function $s(x) = \frac{1}{(x+1)^a} (a>0)$, and note that $s(x)$ is decreasing with respect to $x$, i.e., for all $x_1>k>x_2$, $s(x_1)<\frac{1}{N_k}<s(x_2)$. Therefore, we have
			\begin{align*}
			\sum_{k=0}^K \frac{1}{(k+1)^a} \leq & 1+ \int_0^K s(x) dx = s(0)+ \int_0^K \frac{1}{(x+1)^a} dx\\
			= & \left\{\begin{array}{l}
			1+\ln(K+1),\ a=1\\
			\frac{(K+1)^{-a+1} - a}{1-a},\ a\neq 1.
			\end{array}\right.
			\end{align*}
			Similarly, we have $\sum_{k=0}^K \frac{1}{N_{k+1}} \leq  \left\{\begin{array}{l}
			1/2+\ln(\frac{K}{2}+1),\ a=1\\
			\frac{(K+2)^{-a+1} - (1+a)2 ^{-a}}{1-a},\ a\neq 0
			\end{array}.\right.$
			
			When $a=1$, since $ \ln(K/2+1)\leq \ln(K+1)$, it follows that $$\sum_{k=0}^K \frac{1}{N_{k+1}}\leq \sum_{k=0}^K \frac{1}{N_{k}}\leq 1+\ln(K+1).$$ Plugging the results to \eqref{19} implies \eqref{21-1}.
			
			When $a \neq 1$, we consider the noise term  \begin{align*}
			e_K:=&\frac{\sigma^2\sum_{k=0}^K\left(\frac{2\beta\Gamma_2+\tilde D_3}{N_k}+\frac{2\beta\Gamma_2}{N_{k+1}}\right)}{\beta \cdot (K+1)}.
			\end{align*}
			If $a < 1$, as $K+2 \leq 2(K+1)$, then \begin{equation}\label{23-1}
			\begin{aligned}
			e_K \leq& \sigma^2 \left[\frac{2\Gamma_2+\frac{\tilde D_3}{\beta}}{1-a}\cdot \frac{1}{(K+1)^a}+\frac{2\Gamma_2}{1-a} \cdot \frac{(K+2)^{-a+1}}{K+1}\right]\\
			\leq& \frac{\sigma^2(6\Gamma_2+\frac{\tilde D_3}{\beta})}{1-a}\cdot \frac{1}{(K+1)^a};
			\end{aligned}
			\end{equation}
			If $a>1$, as $K+2 \leq 2(K+1)$ and $1-a<0$, then \begin{equation}\label{23-2}
			\begin{aligned}
			e_K \leq \frac{\sigma^2(3\Gamma_2a+\Gamma_2+\frac{a \tilde D_3}{\beta})}{(a-1)(K+1)}.
			\end{aligned}
			\end{equation}
			By plugging \eqref{23-1} and \eqref{23-2} into \eqref{19}, respectively, we get \eqref{21-2}.
			
			(iii) If $N_k = Q$, where $Q$ is a constant. By \eqref{19}, we have $e_K \leq \frac{\sigma^2(4\beta\Gamma_2+\tilde D_3)}{\beta Q}$, which yields \eqref{22}.
	\end{proof}
	
	\begin{remark}
		Here we make some discussions on Theorem \ref{th-2}. When the sample size increases exponentially, the convergence rate is $\mathcal{O}(\frac{1}{k})$ in expectation, which is comparable to that of the deterministic algorithm in \cite{pmlr-v139-chen21g}. If the sample size increases in a polynomial rate, then the convergence rate corresponds to the parameter $a$ in $[(k+1)^a]$. When $a>1$, the convergence is dominated by the term without noise, which achieves the rate of $\mathcal{O}(\frac{1}{k})$; when $a=1$, the convergence rate is dominated by the noise term, which is bounded by $\mathcal{O}(\frac{\ln k}{k})$ in expectation; when $a<1$, the convergence rate is also dominated by the noise term, which is $\mathcal{O}(\frac{1}{k^a})$ in expectation. When the sample size is fixed, the iterates converge to a neighborhood of the stationary point with constant step sizes. This neighborhood is caused by the stochastic noise and also depends on the sample-size.
	\end{remark}
	
	$\bold{X} = [X_1^\top, \dots, X_N^\top]^\top$ is called an $\epsilon$-stationary point if
	\begin{equation}\label{epsilon}
	\frac{1}{N} E[\|\bold{X}-\hat{\bold X}\|_F^2] \leq \epsilon,\ E[\|\text{grad}f(\hat X)\|_F^2] \leq \epsilon.
	\end{equation}
	Then  based on Theorem \ref{th-2}, we analyze the bound on the number of iterations $K(\epsilon)$ (namely, iteration complexity), the number of sampled gradients (namely, oracle complexity), and the communication cost to achieve an $\epsilon$-stationary point ensuring  that the expected squared norm of stochastic gradient $E[\|\text{grad}f(\hat X)\|_F^2] \leq \epsilon
	$, where $\epsilon$ shows the algorithm accuracy. We formally state the results in the following corollary.

	\begin{corollary}\label{co-1}
	With the same conditions of Theorem \ref{th-2}, and let $C:=\frac{\sigma^2(16+\Gamma_0 \alpha^2\delta_1^2)(4\beta\Gamma_2+\tilde D_3)}{\beta Q}$. Denote by $|\mathcal{E}|$ the number of network edges. The iteration and oracle complexity, and the number of communications required to obtain an  $\epsilon$ (or $\epsilon +C$)-stationary point for the variable sample sizes are shown in Table 2.
	
	\renewcommand{\arraystretch}{1.2}
	\vspace{-0.5cm}
	\begin{table}[h!]
		\centering
		\caption{Iteration and oracle complexity, and communications }
	   	\begin{tabularx}{\columnwidth}{c|c|c|c}
			\hline
			$N_k$& iteration& oracle  & communication\\ \hline
			$[q^{-k}](q \in(0,1))$ & $\mathcal{O}(\frac{1}{\epsilon})$ & $\mathcal{O}(e^{\frac{c_1}{\epsilon} \ln (1/q)})$ & $\mathcal{O}(\frac{t|\mathcal{E}|}{\epsilon})$\\ \hline
			$[(k+1)^a](a=1)$ & $\mathcal{O}(\frac{1}{\epsilon}\ln\frac{1}{\epsilon})$ & $\mathcal{O}(\frac{1}{\epsilon^2}\ln^2\frac{1}{\epsilon})$ &$\mathcal{O}(\frac{t|\mathcal{E}|}{\epsilon}\ln\frac{1}{\epsilon})$ \\ \hline
			$[(k+1)^a](a<1)$& $\mathcal{O}(\frac{1}{\epsilon^{1/a}})$& $\mathcal{O}(\frac{1}{\epsilon^{1+1/a}})$ &$\mathcal{O}(\frac{t|\mathcal{E}|}{\epsilon^{1/a}})$\\ \hline
			$[(k+1)^a](a>1)$& $\mathcal{O}(\frac{1}{\epsilon})$ & $\mathcal{O}(\frac{1}{\epsilon^{a+1}})$ & $\mathcal{O}(\frac{t|\mathcal{E}|}{\epsilon})$\\ \hline
			$Q>0$  & $\mathcal{O}(\frac{1}{\epsilon})$ & $\mathcal{O}(\frac{Q}{\epsilon})$ & $\mathcal{O}(\frac{t|\mathcal{E}|}{\epsilon})$\\ \hline
		\end{tabularx}
	\end{table}
	\end{corollary}
	\begin{proof}
	(i) $N_k = [q^{-k}]$: Based on Theorem \ref{th-2}, there exists $c_1>0$ such that $\min_{k=0,\dots, K}E[\|\text{grad}f(\hat X_k)\|_F^2] \leq \frac{c_1}{K+1}$. For any $k +1 \geq K(\epsilon) + 1= \frac{c_1}{\epsilon}$, we obtain $\min_{k=0,\dots, K}E[\|\text{grad}f(\hat X_k)\|_F^2] \leq \epsilon$. Then, the bound of the oracle complexity is derived by \[\sum_{k=0}^{K(\epsilon)} N_k = \sum_{k=0}^{K(\epsilon)} q^{-k} \leq \frac{q^{-(K(\epsilon)+1)}}{q^{-1}  - 1} \leq \frac{1}{1-q}q^{-\frac{c_1}{\epsilon}}= \frac{1}{1-q}e^{\frac{c_1}{\epsilon} \ln (1/q)}.\]
	
	(ii) $N_k = [(k+1)^a]$: For $a=1$, by Theorem \ref{th-2}, there exists $c_2>0$ such that $\min_{k=0,\dots, K}E[\|\text{grad}f(\hat X_k)\|_F^2] \leq \frac{c_2\ln (K+1)}{K+1}$. We achieve an $\epsilon$-stationary solution for any $k +1\geq K(\epsilon) +1= O(\frac{1}{\epsilon} \ln \frac{1}{\epsilon})$, which then implies \[ \sum_{k=0}^{K(\epsilon)} (k+1) = \frac{(K(\epsilon)+2)(K(\epsilon)+1)}{2} \leq (K(\epsilon)+1)^2 = \mathcal{O}( \frac{1}{\epsilon^2}\ln^2\frac{1}{\epsilon}). \]
	Similarly, for $a< 1$ and $a>1$, there exist $c_3, c_4 >0$ such that $\min_{k=0,\dots, K}E[\|\text{grad}f(\hat X_k)\|_F^2] \leq \frac{c_3}{(K+1)^a}$ and $\min_{k=0,\dots, K}E[\|\text{grad}f(\hat X_k)\|_F^2] \leq \frac{c_4}{K+1}$, respectively. Then, the corresponding iteration bound is $K(\epsilon ) +1= (\frac{c_3}{\epsilon})^{1/a}$ and $\frac{c_4}{\epsilon}$. By plugging these results into \[\sum_{k=0}^{K(\epsilon)} (k+1)^a\leq \int_0^{K(\epsilon)+1} (x+1)^a dx\leq  \frac{(K(\epsilon)+2)^{a+1}}{a+1}, \] we derive the bound of the oracle number, respectively.
	
	(iii) $N_k = Q>0$: By Theorem \ref{th-2}, there exists $c_5>0$ such that $\min_{k=0,\dots, K}E[\|\text{grad}f(\hat X_k)\|_F^2] \leq \frac{c_5}{K+1}+C$, which implies that $K(\epsilon ) +1= \frac{c_5}{\epsilon}$ and $\sum_{k=0}^{K(\epsilon)} N_k = \frac{c_5 Q}{\epsilon}$.
	
	The multi-agent system requires $2t|\mathcal{E}|$ communication rounds in one iteration, so multiplying the iteration complexity by $2t|\mathcal{E}|$ directly derives the total number of communications.
	\end{proof}
	
	\begin{remark}
	Though the variable sampling scheme can improve the convergence rate and performance of Algorithm \ref{alg:1}, we cannot increase the sample size blindly, since it leads to high oracle complexity. According to Table 2 in Corollary \ref{co-1}, $N_k = [(k+1)^a](a>1)$ should be the optimal choice as the oracle complexity is not extremely high while ensuring the best iteration complexity $\mathcal{O}(\frac{1}{\epsilon})$ among all settings.  Compared with   the DRSGD scheme \cite{pmlr-v139-chen21g}, our scheme can  significantly reduce  the number of projection and retraction computation steps in \eqref{4} and \eqref{v}, as well as the communication rounds, while slightly increase the number of sampling. Therefore, the results present a trade-off between the sample size and communication costs, and demonstrate that our algorithm is competitive in the scenarios where computation and communication costs are expensive   than the gradient samplings.
	\end{remark}
	
	\section{Numerical Experiments}\label{sec-5}
	
			\begin{figure*} [!h]
				\centering
				\setlength{\abovecaptionskip}{0cm}
				\includegraphics[width=2\columnwidth]{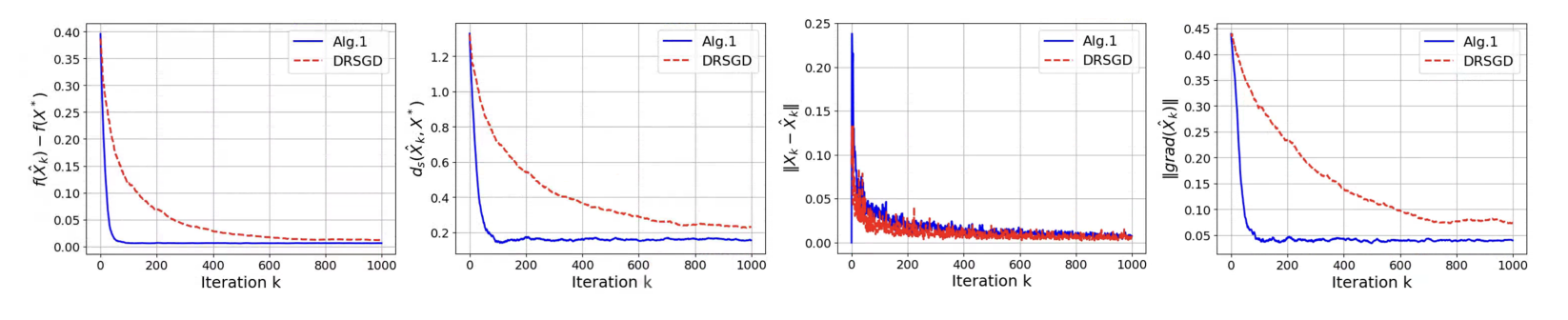}
				\caption{Numerical results on synthetic dataset with DRSGD and Alg.\ref{alg:1}, eigengap $\triangle = 0.8$, Graph:ring, t= 1.}
				\label{fig1}
			\end{figure*}	

	In this section, we conduct some experiments to show the empirical performance of Algorithm \ref{alg:1} on the principal component analysis problems.
	
	Principal component analysis (PCA) dimension reduction is with extensive applications in machine learning. It can be formulated as follows:
	\begin{equation}\label{6-1}
	\begin{aligned}
	\min_{X_i \in \text{St}(n,r)}&\ -\frac{1}{2N}\sum_{i=1}^N X_i^\top A_i^\top A_i X_i,\\
	s.t.&\ X_1=\dots=X_N,
	\end{aligned}
	\end{equation}
	where $A_i \in \mathbb{R}^{m_i \times n}$ is the local data matrix of agent $i$ and $m_i$ is the data size. Denote by $A:=[A_1^\top, \dots, A_N^\top]$ the global data matrix. We also utilize the Bluefog \cite{bluefog} package to conduct a distributed structure and Pymanopt \cite{townsend2016pymanopt} toolbox to compute on the Stiefel manifold.
	
	We run $N = 4$ CPU nodes and fix $m_1 = \dots = m_4 = 2500$, $n = 8$ and $r = 3$. We generate the data matrix $A$ including $2500N \times n$ i.i.d samples following standard multi-variate Gaussian distribution, and $A_i$ is obtained by distributing $A$ into the $4$ CPU nodes.
	As $X^*Q$ is also a solution of \eqref{6-1} with an orthogonal matrix $Q \in \mathbb{R}^{r\times r}$, we define by $d_s(X,X^*) = \min_{Q \in \mathbb{R}^{r\times r}, Q^\top Q = Q Q^\top = I_r} \|XQ - X^*\|$ to measure the distance between $X$ and $X^*$.
	The algorithms are measured by four metrics, i.e., the objective function $f(\hat X_k)- f^*$, the distance to the global optimum $d_s(X_k,X^*)$, the consensus error $\|\bold{X}_k - \bold{ \hat X}_k\|$, and the gradient norm $\|\text{grad}f(\hat X_k)\|$, respectively.
	
	 We then compare our algorithm with the state-of-art distributed stochastic optimization algorithm DRSGD in \cite{pmlr-v139-chen21g} and receive the results in Fig.\ref{fig1}. Here we let the communication graph among the four agents to be ring graph.  Furthermore, we set the variable sample size of Alg.\ref{alg:1} to be $N_k = k+1$ which increases in a polynomial rate. Let the multi-steps of consensus $t=1$, the step sizes $\alpha = 1$, and $\beta = 0.1$ for Alg.\ref{alg:1} and be diminishing for DRSGD, respectively. We set the maximum iteration rounds to $1000$ and run Alg.\ref{alg:1}. The results can be seen in Fig.\ref{fig1} which illustrate that Alg.\ref{alg:1} converges faster than DRSGD. Although we have equipped a multi-step communication to achieve the theoretic consensus of the algorithm, the empirical results in Fig.\ref{fig1} show that both of the algorithms achieve consensus with only one communication step in practice.
	 
	 		\begin{figure} [!h]
	 			\centering
	 			\setlength{\abovecaptionskip}{0cm}
	 			\includegraphics[width=1.9in]{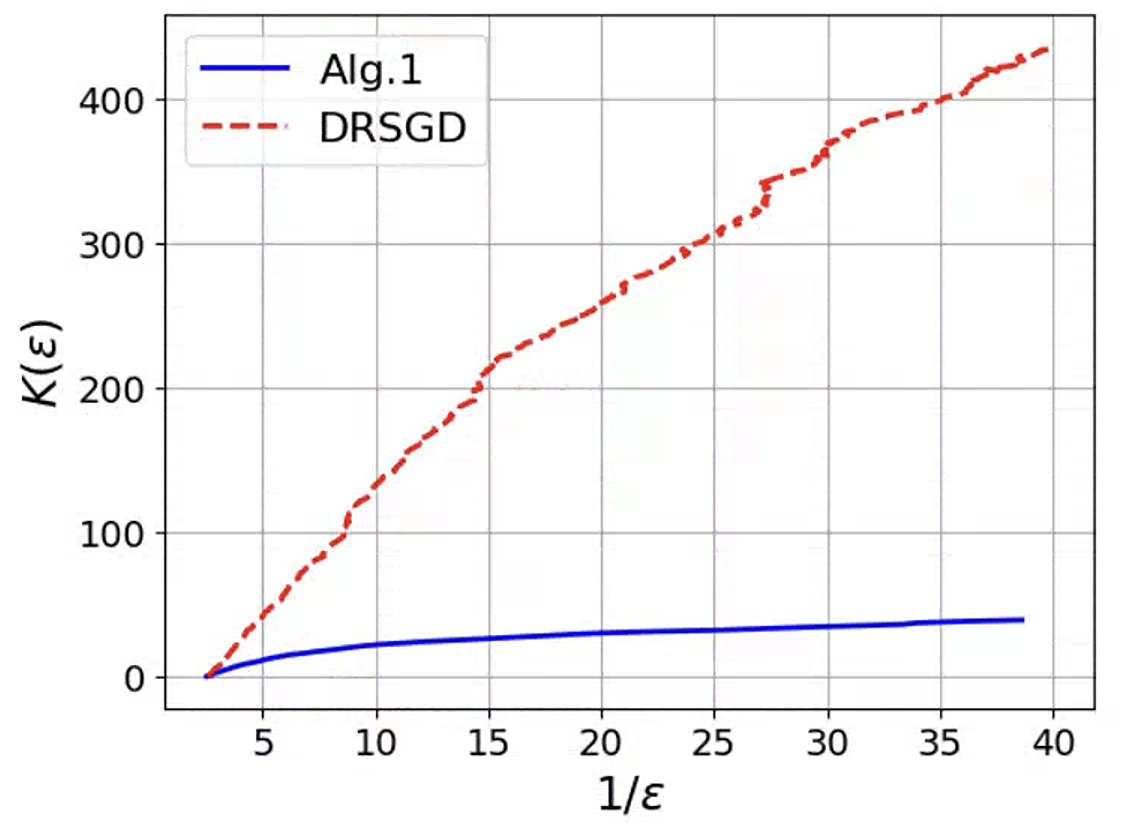}
	 			\caption{Iteration complexity of Alg.\ref{alg:1} and DRSGD}
	 			\label{fig3}
	 		\end{figure}
	 		
	 		\begin{figure}[!h]
	 			\centering
	 			\setlength{\abovecaptionskip}{0cm}
	 			\includegraphics[width=1.9in]{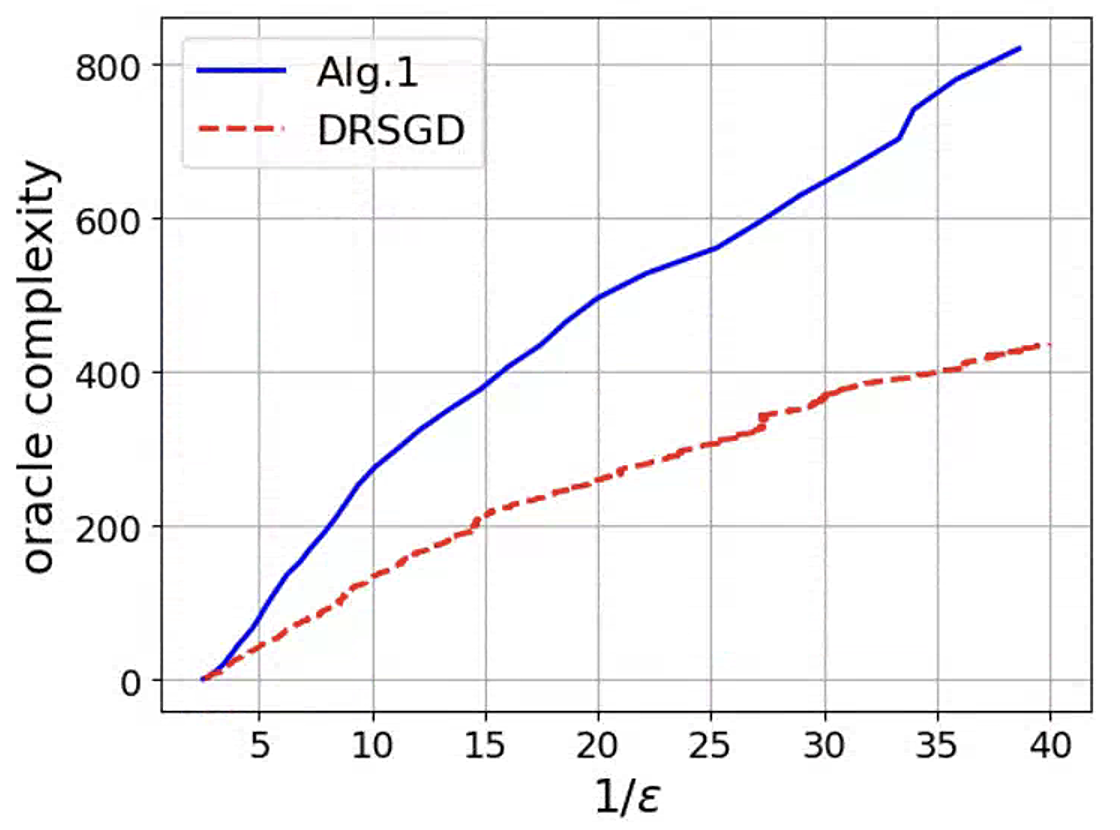}
	 			\caption{Oracle complexity of Alg.\ref{alg:1} and DRSGD}
	 			\label{fig4}
	 		\end{figure}

	We further compare the iteration and oracle complexity of the two algorithms under the above settings. As shown in Fig.\ref{fig3}, to achieve an $\epsilon$-stationary point with the same accuracy, Alg.\ref{alg:1} requires much less iteration rounds than DRSGD, which also means it can significantly reduce the communication rounds. However, Fig.\ref{fig4} shows that Alg.\ref{alg:1} needs more sampled gradients than DRSGD. The results illustrate that our algorithm can significantly reduce the communication and computation costs by partly sacrificing the amount of gradient samples.
	
	In addition, we conduct experiment to see the impact of the sample sizes. We run Alg.\ref{alg:1} with step sizes $\alpha  = 1, \beta = 0.1$ and let $N_k =1, k+1, [0.9^{-k}]$ and $N_k = [0.85^{-k}], [0.9^{-k}], [0.95^{-k}]$, respectively. The performance of algorithm for both cases are displayed in Fig.\ref{fig2} and Fig.\ref{fig5}  which illustrate that a faster increasing sample size leads to a better convergence rate and stability, meanwhile costs more sampled data and heavier computations per one iteration.
		\begin{figure} [!h]
			\centering
			\setlength{\abovecaptionskip}{0cm}
			\includegraphics[width=1.9in]{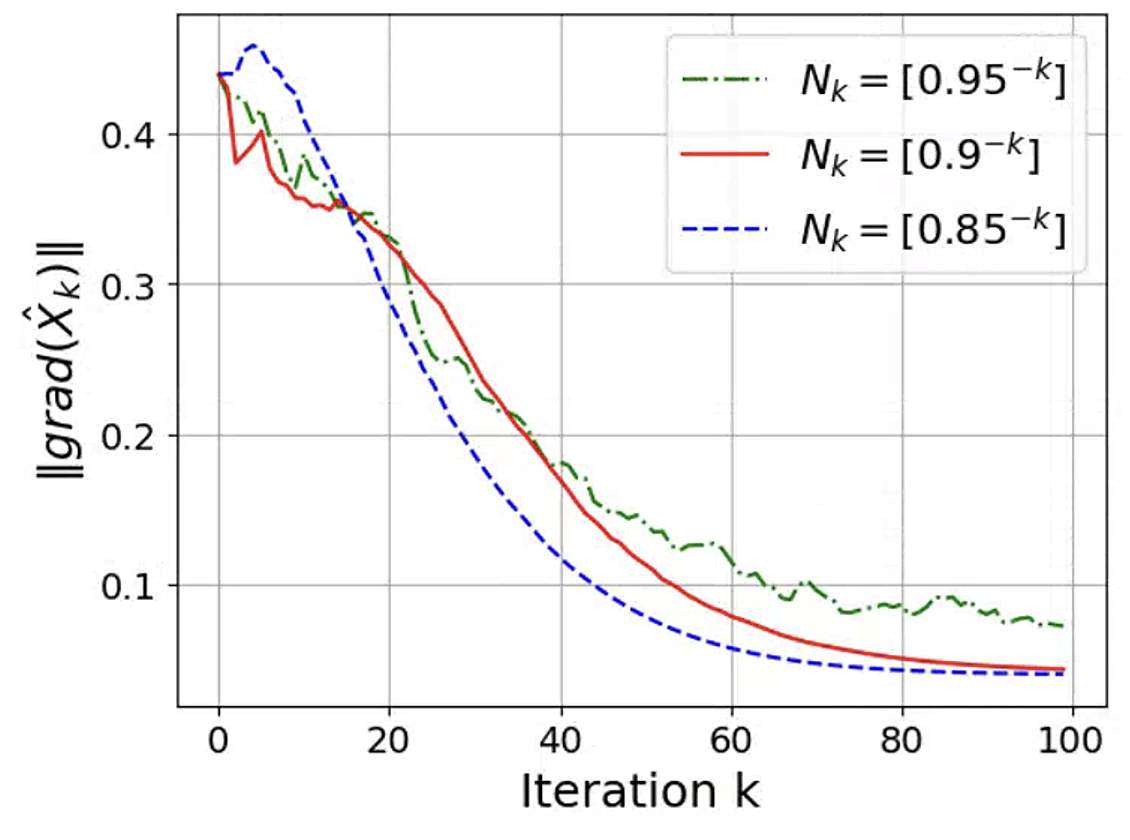}
			\caption{The performance of Alg.\ref{alg:1} with $N_k = [0.85^{-k}], [0.9^{-k}], [0.95^{-k}]$ under a ring graph.}
			\label{fig2}
		\end{figure}
		\begin{figure} [!h]
			\centering
			\setlength{\abovecaptionskip}{0cm}
			\includegraphics[width=2.0in]{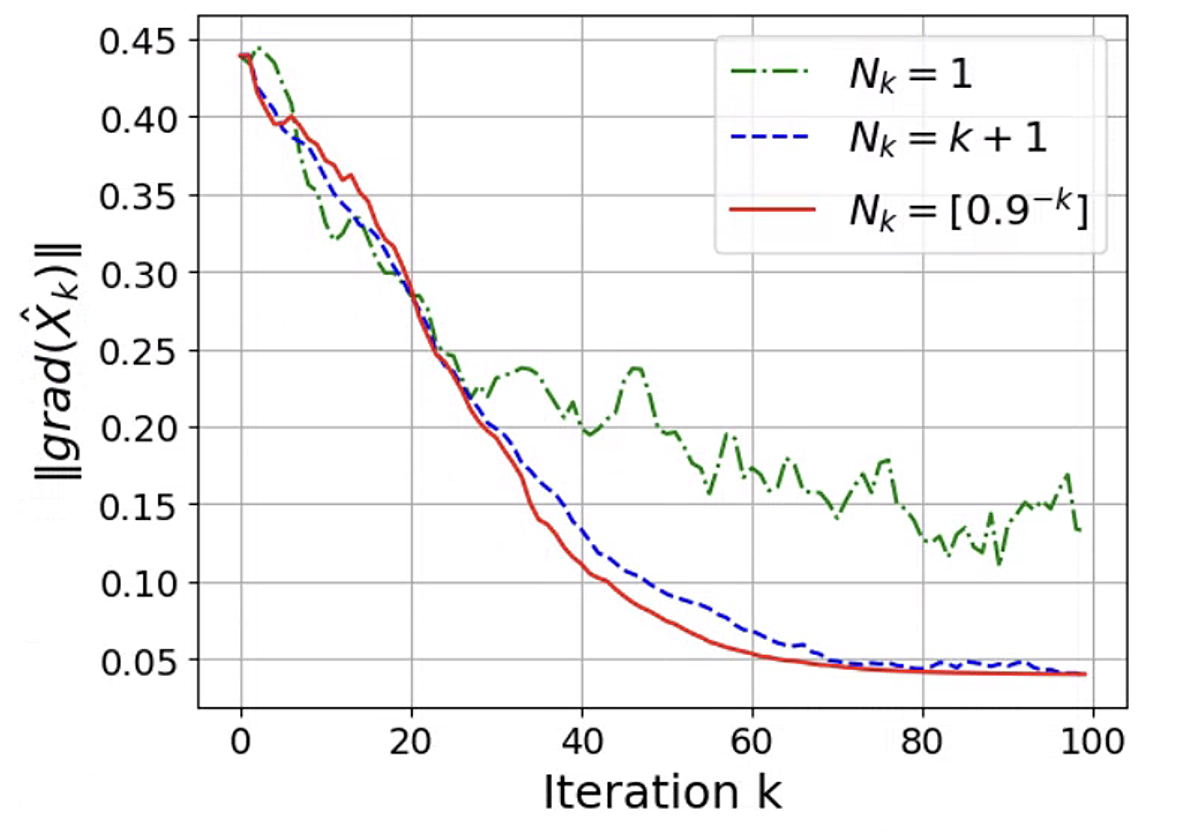}
			\caption{The performance of Alg.\ref{alg:1} with $N_k =1, k+1, [0.9^{-k}]$ under a ring graph.}
			\label{fig5}
		\end{figure}
	
	\section{Conclusion}\label{sec-6}

In this paper, we have proposed a distributed Riemannian stochastic gradient tracking algorithm with variable sample sizes for optimization on the Stiefel manifold over connected networks. If the agents are set to start from a suitably defined local region, we prove that the iterates of all agents always remain in this region and converge to a stationary point (or neighborhood) with fixed step sizes in expectation. We further show that the convergence rate can be affected by the increasing sample size, which can reduce the noise variance. The convergence rate of the iterates with a exponentially increasing sample size is comparable to the non-stochastic framework on the Stiefel manifold. We also give analysis on the convergence rate with polynomially increasing sample size and a constant sample size. Based on the convergence results, we establish the iteration, oracle, and communication complexity, and present the trade-off between communication costs and gradient sampling. Finally, we conduct numerical experiments that demonstrate the effectiveness of the algorithm and the theoretical results. The extension to nonsmooth framework or general compact submanifolds are promising future research directions.

\bibliographystyle{plain}
\bibliography{reference}

\appendix
\section{Proof of Lemma \ref{lem-5}}\label{app_1}
		(i) We begin with analyzing the gradient tracking error of $k+1$.  Denote $J = \frac{1}{N}\bm{1}_N\bm{1}_N^\top$. According to Lemma \ref{gt}, we have \begin{align*}
		\bold{Y}_{k+1} - \bar{\bold{F}}_{k+1} =& (I_N \otimes I_n)\bold{Y}_{k+1} - (J \otimes I_n)\bold{Y}_{k+1} \\
		=& ((I_N - J) \otimes I_n)\bold{Y}_{k+1}.
		\end{align*}
		By plugging \eqref{g-t} in the above inequality, it follows that\begin{align*}
		&\bold{Y}_{k+1} - \bar{\bold{F}}_{k+1} \\
		=&((I_N - J) \otimes I_n)((W^t \otimes I_n)\bold{Y}_k+ F(\bold{X}_{k+1}) - F(\bold{X}_{k}))\\
		=&((W^t - J) \otimes I_n)\bold{Y}_k+((I_N - J) \otimes I_n)(F(\bold{X}_{k+1}) - F(\bold{X}_{k})),
		\end{align*}where the last equality follows by the doubly stochastic of $W^t$ from Assumption \ref{assu-1} (2) and Remark \ref{r-1}. Then by using Lemma \ref{gt}, we have \begin{align*}
		&\bold{Y}_{k+1} - \bar{\bold{F}}_{k+1} \\
		=&(W^t \otimes I_n)\bold{Y}_k-\bar {\bold{F}}_k+((I_N - J) \otimes I_n)(F(\bold{X}_{k+1}) - F(\bold{X}_{k})).
		\end{align*}
		
		Take the norm of the equation above, it follows that
		\begin{equation}\label{1-6}
		\begin{aligned}
		\|\bold{Y}_{k+1} - \bar{\bold{F}}_{k+1} \|_F\leq &\|(W^t \otimes I_n)\bold{Y}_k-\bar {\bold{F}}_k\|_F\\
		+&\|((I_N - J) \otimes I_n)(F(\bold{X}_{k+1}) - F(\bold{X}_{k}))\|_F.
		\end{aligned}
		\end{equation}
		Recalling the fact that $\sigma_2$ is the spectral norm of $W - J$ stated in Remark \ref{r-1}, we have\begin{equation}\label{1-7}
		\begin{aligned}
		\|(W^t \otimes I_n)\bold{Y}_k-\bar {\bold{F}}_k\|_F = &\|((W^t - J) \otimes I_n) (\bold{Y}_k - \bar {\bold{F}}_k)\|_F \\
		\leq& \|W^t - J\|_F\| (\bold{Y}_k - \bar {\bold{F}}_k)\|_F \\
		\leq& \sigma_2^t\|\bold{Y}_k - \bar{\bold{F}}_{k}\|_F.
		\end{aligned}
		\end{equation}
		Substituting \eqref{1-7} to \eqref{1-6} yields
		\begin{equation}\label{1-5}
		\|\bold{Y}_{k+1} - \bar{\bold{F}}_{k+1} \|_F\leq \sigma_2^t\|\bold{Y}_k - \bar{\bold{F}}_{k}\|_F+\|F(\bold{X}_{k+1}) - F(\bold{X}_{k})\|_F.
		\end{equation}
		Applying the iteration Lemma \ref{al-5} to \eqref{1-5}, we finally prove \eqref{l-2}.
		
		(ii) We begin with presenting the one step improve of consensus. If $\bold{X}_k \in \mathcal{S}$, using the definition of IAM in \eqref{iam}, we have$$\|\bold{X}_{k+1} - \bold{\hat X}_{k+1}\|_F^2 \leq \|\bold{X}_{k+1} - \bold{\hat X}_{k}\|_F^2.$$
	 Since $X_{i,k} \in \text{St}(n,r)$ and $-\alpha \text{grad}\ h_{i,t}(\bold{X}_k)-\beta v_{i,k} \in T_{X_{i,k}}\mathcal{M}$, by using \eqref{re-u} and the property of retraction in Lemma \ref{lem-1}, we get
		\begin{align*}
		\|\bold{X}_{k+1} - \bold{\hat X}_{k+1}\|_F^2= &\sum_{i=1}^N\|\mathcal{R}_{X_{i,k}}(-\alpha \text{grad}\ h_{i,t}(\bold{X}_k)-\beta v_{i,k})-\hat X_k\|_F^2\\
		\leq& \sum_{i=1}^N\|X_{i,k}-\alpha \text{grad}\ h_{i,t}(\bold{X}_k)-\beta v_{i,k}-\hat X_k\|_F^2\\
		= &\|\bold{X}_{k}-\bold{\hat X}_k-\alpha \text{grad}\ h_t(\bold{X}_k)-\beta \bold{v}_{k}\|_F^2.
		\end{align*}
	
		Since $v_{i,k} = \mathcal{P}_{T_{X_{i,k}} \mathcal{M}} Y_{i,k}$, by the nonexpansiveness of the orthogonal projection, we obtain $\|\bold{v}_k\|_F \leq \|\bold{Y}_k\|_F$, which implies
		\begin{equation}\label{1-8}
		\begin{aligned}
		\|\bold{X}_{k+1} - \bold{\hat X}_{k+1}\|_F \leq&\|\bold{X}_{k}-\bold{\hat X}_k-\alpha \text{grad}\ h_t(\bold{X}_k)-\beta\bold{v}_{k}\|_F\\
		\leq &\|\bold{X}_{k}-\bold{\hat X}_k-\alpha \text{grad}\ h_t(\bold{X}_k)\|_F+\beta\|\bold{v}_{k}\|_F\\
		\leq &\|\bold{X}_{k}-\bold{\hat X}_k-\alpha \text{grad}\ h_t(\bold{X}_k)\|_F+\beta\|\bold{Y}_{k}\|_F.
		\end{aligned}
		\end{equation}
		Suppose $\bold{X}_k \in \mathcal{S}$, by Lemma \ref{al-6}, we have  \begin{align*}
		&\|\bold{X}_{k}-\bold{\hat X}_k-\alpha \text{grad}\ h_t(\bold{X}_k)\|_F^2\\
		=&\|\bold{X}_{k}-\bold{\hat X}_k\|_F^2+\|\alpha \text{grad}\ h_t(\bold{X}_k)\|_F^2-2\alpha \langle \bold{X}_{k}-\bold{\hat X}_k,\text{grad}\ h_t(\bold{X}_k) \rangle\\
		\leq & \|\bold{X}_{k}-\bold{\hat X}_k\|_F^2+\left(\alpha^2 - \frac{\alpha \Phi}{L_t}\right)\|\text{grad}\ h_t(\bold{X}_k) \|_F^2.
		\end{align*}
		By utilizing \eqref{9-2} in Lemma \ref{al-3}, it follows that
		\begin{equation*}\label{1-9}
		\begin{aligned}
		\|\bold{X}_{k}-\bold{\hat X}_k-\alpha \text{grad}\ h_t(\bold{X}_k)\|_F^2\leq  (1-(\alpha L_t \Phi - \alpha^2 L_t^2)) \|\bold{X}_{k}-\bold{\hat X}_k\|_F^2.
		\end{aligned}
		\end{equation*}
		Let $\rho_t^2:=1-(\alpha L_t \Phi - \alpha^2 L_t^2)$, and substituting the above result into \eqref{1-8} implies \eqref{l-3}.
		Notice that $\rho_t \leq 1$ since $\alpha \leq \frac{\Phi}{L_t}$.
		
		By applying Lemma \ref{al-5} to \eqref{l-3}, we finally prove \eqref{l-4}.
		
		\rightline{$\qedsymbol$}
\section{Proof of Lemma \ref{lem-3}}\label{app_2}
	We prove this by induction. When $k=0$, by Assumption \ref{assu-3} and Remark \ref{re-4}, and using \eqref{v-s}, one has $\|Y_{i,0}\|_F = \|F_i(X_{i,0})\| \leq A$. Let $\bar F_{-1} = \bar Y_0$. Then for each $ i \in \mathcal{N}$,
	$$\|Y_{i,0} - \bar F_{-1}\|_F \leq \|Y_{i,0}\|_F+\|\bar F_{-1}\|_F \leq 2A,$$ while $\bold{X}_0 \in \mathcal{S}$ by the initial selection.
	
	Suppose there exists a $k_0$, that for every $0 \leq k \leq  k_0$, we have $\|Y_{i,k}\|_F\leq 10A+L_G$ and $\|Y_{i,k} - \bar F_{k-1}\|_F \leq 10A+L_G$. It also satisfies that $\bold{X}_{k} \in \mathcal{S}$.
	
	Next, we present the analysis when $k=k_0+1$. We first prove that $\bold{X}_{k_0+1} \in \mathcal{S}$. The result is induced by a similar argument in the proof of \cite[Lem. 4.1]{pmlr-v139-chen21g}.   
	Based on the induction hypothesis $\bold{X}_{k_0} \in \mathcal{S}$, together with \eqref{l-3} and $\beta \leq \frac{1-\rho_t}{10A+L_G}\delta_1$, we have
	\begin{equation}
	\begin{aligned}
	\|\bold{X}_{k_0+1} - \bold{\hat X}_{k_0+1}\|_F \leq&\rho_t\|\bold{X}_{k_0}-\bold{\hat X}_{k_0}\|_F+\beta \|\bold Y_{k_0}\|_F\\
	\leq & \sqrt{N}\delta_1,
	\end{aligned}
	\end{equation}
	which implies $\bold{X}_{k_0+1} \in \mathcal{S}_1$.
	
	Since $\|v_{i,k_0}\|_F \leq \|Y_{i,k_0}\|_F \leq 10A +L_G$ and $\beta \leq \frac{\alpha \delta_1}{5(10A+L_G)}$, we derive \begin{equation}\label{a-2}
	\begin{aligned}
	&\|X_{i,k_0+1}-\hat{X}_{k_0+1}\|_F\\
	\leq &\|X_{i,k)+1}-\hat{X}_{k_0}\|_F+\|\hat X_{k_0} - \hat{X}_{k_0+1}\|_F\\
	\leq &\|X_{i,k_0}-\alpha \text{grad}h^t_i(\bold{X}_{k_0})-\beta v_{i,k_0}-\hat{X}_{k_0}\|_F+\|\hat X_{k_0} - \hat{X}_{k_0+1}\|_F\\
	\leq & \|X_{i,k_0}-\alpha \text{grad}h^t_i(\bold{X}_{k_0})-\hat{X}_{k_0}\|_F+\frac{\alpha \delta_1}{5}+\|\hat X_{k_0} - \hat{X}_{k_0+1}\|_F
	\end{aligned}
	\end{equation}
	For the last term, since the derivation of \eqref{3-18} only uses Lemma \ref{lem-1} and Lemma \ref{al-3}, by \eqref{3-18} and \eqref{al-2}, we get
	\begin{equation}
	\begin{aligned}
	\|\hat X_{k_0+1}-\hat X_{k_0}\|_F\leq &\frac{2}{\bar \sigma_{r,k_0+1}+\bar \sigma_{r,k_0}}(\frac{2M\alpha^2L_t^2+\alpha L_t}{N}\|\bold{X}_{k_0}-\bold{\hat X}_{k_0}\|_F^2\\
	+&\frac{2M\beta^2}{N}\|\bold{v}_{k_0}\|_F^2+\beta\|\bar v_{k_0}\|_F)\\
	\leq & \frac{2}{\bar \sigma_{r,k_0+1}+\bar \sigma_{r,k_0}}[(2M\alpha^2L_t^2+\alpha L_t)\delta_1\\
	+&\beta (10A+L_G)+2M\beta^2(10A+L_G)^2].
	\end{aligned}
	\end{equation}
	where the last inequality follows by $\| v_{i,k_0}\|_F \leq \|Y_{i,k_0}\|_F \leq 10A +L_G$.
	And since $L_t\leq 2,\ \alpha \leq 1,\ \beta \leq \frac{\alpha \delta_1}{5(10A+L_G)}$, we have \begin{equation}\label{a-1}
	\begin{aligned}
	&\|\hat X_{k_0+1}-\hat X_{k_0}\|_F\\
	\leq &\frac{2}{\bar \sigma_{r,k_0+1}
		+\bar \sigma_{r,k_0}}(\frac{252}{25}\alpha \delta_1^2+\frac{\alpha \delta_1}{5})\\
	\leq & \frac{2}{\bar \sigma_{r,k_0+1}+\bar \sigma_{r,k_0}}(\frac{252}{625r}\alpha \delta_2^2+\frac{\alpha \delta_2}{25 \sqrt{r}}),
	\end{aligned}
	\end{equation}
	where the last inequality is derived from $\delta_1 \leq \frac{1}{5\sqrt{r}}\delta_2$.
	
	According to \eqref{1}, we get $$\begin{aligned}
	\text{grad}h_{i,t}(\bold{X}) = &X_i-\sum_{j=1}^N W_{ij}^t X_j -\frac{1}{2}X_i\sum_{j=1}^N W_{ij}^t(X_i-X_j)^\top(X_i-X_j),
	\end{aligned}$$
	and then we derive
	\begin{align}\label{a-3}
	&\|X_{i,k_0}-\alpha \text{grad}h_{t,i}(\bold{X}_{k_0})-\hat X_{k_0}\|_F \notag\\
	=&\|(1-\alpha)(X_{i,k_0}-\hat{X}_{k_0})+\alpha (\Bar{X}_{k_0}- \hat{X}_{k_0})+\alpha \sum_{j=1}^N W_{ij}^t(X_{j,k_0}-\bar X_{k_0})\notag\\
	+&\frac{\alpha}{2}X_{i,k_0}\sum_{j=1}^N W_{ij}^t(X_{i,k_0}-X_{j,k_0})^\top(X_{i,k_0}-X_{j,k_0})\|_F\notag\\
	\leq & (1-\alpha)\delta_2+\alpha \|\Bar{X}_{k_0}- \hat{X}_{k_0}\|_F+\alpha\|\sum_{j=1}^N (W_{ij}^t-\frac{1}{N})X_{j,k_0}\|_F\\
	+&\frac{1}{2}\|\alpha \sum_{j=1}^N W_{ij}^t(X_{i,k_0}-X_{j,k_0})^\top(X_{i,k_0}-X_{j,k_0})\|_F\notag\\
	\leq&(1-\alpha)\delta_2+2\alpha \delta_1^2 \sqrt{r}+\alpha\|\sum_{j=1}^N (W_{ij}^t-\frac{1}{N})X_{j,k_0}\|_F+2\alpha \delta_2^2\notag\\
	\leq & (1-\frac{\alpha}{2})\delta_2+2\alpha \delta_1^2 \sqrt{r}+2\alpha \delta_2^2.\notag
	\end{align}

	Substituting \eqref{a-1} and \eqref{a-3} into \eqref{a-2} we have
	\begin{align}
	&\|X_{i,k_0+1}-\hat{X}_{k_0+1}\|_F\\
	\leq &(1-\frac{\alpha}{2})\delta_2+2\alpha \delta_1^2 \sqrt{r}+2\alpha \delta_2^2+\frac{\alpha \delta_1}{5}\\
	+&\frac{2}{\bar \sigma_{r,k_0+1}+\bar \sigma_{r,k_0}}(\frac{252}{625r}\alpha \delta_2^2+\frac{\alpha \delta_2}{25 \sqrt{r}}).
	\end{align}
	Since $\bold{X}_{k_0} \in \mathcal{S}_1$ and $\bold{X}_{k_0+1} \in \mathcal{S}_1$, we have $$\sigma_r(\Bar{X}_{k_0}) \geq 1-2\frac{\|\bold{X}_{k_0} - \hat{\bold{X}}_{k_0}\|_F^2}{N}\geq 1-2\delta_1^2>0$$  and $\bar \sigma_{r,k_0+1}\geq 1-2\delta_1^2$, which imply that \begin{equation}\label{a-4}
	\begin{aligned}
	&\|X_{i,k_0+1}-\hat{X}_{k_0+1}\|_F\\   
	\leq &(1-\frac{\alpha}{2})\delta_2+2\alpha \delta_1^2 \sqrt{r}+2\alpha \delta_2^2+\frac{\alpha \delta_1}{5}\\
	+&\frac{1}{1-2\delta_1^2}(\frac{252}{625r}\alpha \delta_2^2+\frac{\alpha \delta_2}{25 \sqrt{r}}).
	\end{aligned}
	\end{equation}
	substituting the conditions of $\delta_1,\ \delta_2$ into \eqref{a-4} we have $$\|X_{i,k_0+1}-\hat{X}_{k_0+1}\|_F \leq \delta_2.$$Therefore, we prove that $\bold{X}_{k_0+1}\in \mathcal{S}_2$.
	
	In summary, we have shown that $X_{k_0+1} \in \mathcal{S}_1$ and $X_{k_0+1} \in  \mathcal{S}_2$. According to Definition \ref{l-r} that $\mathcal{S} = \mathcal{S}_1 \bigcap \mathcal{S}_2 $, we then get $\bold{X}_{k_0+1}\in \mathcal{S}$.
	
	Then, by \eqref{g-t} and since $W^t$ is doubly stochastic according to Remark \ref{r-1}, we have \begin{align*}
	&\|Y_{i,k_0+1} - \bar F_{k_0}\|_F \\
	=& \|\sum_{j=1}^N W_{ij}^t Y_{j,k_0}-\bar F_{k_0}+F_i(X_{i,k_0+1}) - F_i(X_{i,k_0})\|_F\\
	\leq & \|\sum_{j=1}^N \left(W_{ij}^t - \frac{1}{N}\right)(Y_{j,k_0}-\bar F_{k_0-1})\|_F+\|F_i(X_{i,k_0+1}) - F_i(X_{i,k_0})\|_F\\
	\leq & \sigma_2^t\sqrt{N}\|Y_{j,k_0}-\bar F_{k_0-1}\|_F+\|F_i(X_{i,k_0+1}) - F_i(X_{i,k_0})\|_F.
	\end{align*}
	By the induction hypothesis that $\|Y_{i,k} - \bar F_{k-1}\|_F \leq 10A+L_G$ for any $k\leq k_0$, it follows that
	\begin{align}	\label{2-5}
	&\|Y_{i,k_0+1} - \bar F_{k_0}\|_F \notag\\
	\leq& \sigma_2^t \sqrt{N}(10A+L_G)+\|\text{grad}f(X_{i,k_0+1}) - \text{grad}f(X_{i,k_0})\|_F\\
	+&\|F_i(X_{i,k_0+1}) - \text{grad} f_i(X_{i,k_0+1})\|_F+\|F_i(X_{i,k_0}) - \text{grad} f_i(X_{i,k_0})\|_F.\notag
	\end{align}
	
	By recalling that $\max_{X \in \text{St}(n,r)}\|\text{grad} f_i(X)\|_F \leq A$ stated in Remark \ref{re-4}, we derive \begin{equation}\label{2-6}
	\|F_i(X_i) - \text{grad} f_i(X_i)\|_F \leq \|F_i(X_i) \|_F+\|\text{grad} f_i(X_i)\|_F\leq 2A.
	\end{equation}
	Then by using Lemma \ref{al-4}, and substituting \eqref{2-6} into \eqref{2-5}, we further have$$\|Y_{i,k_0+1} - \bar F_{k_0}\|_F\leq\sigma_2^t \sqrt{N}(10A+L_G)+L_G\|X_{i,k_0+1} - X_{i,k_0}\|_F+4A.$$Moreover, by using $\sigma_2^t \leq \frac{1}{2\sqrt{N}}$ since $t \geq [\log_{\sigma_2}(\frac{1}{2\sqrt{N}})]$, we get \begin{equation}\label{2-3}
	\|Y_{i,k_0+1} - \bar F_{k_0}\|_F\leq\frac{1}{2}(10A+L_G)+4A+L_G\|X_{i,k_0+1} - X_{i,k_0}\|_F.
	\end{equation}
	
	For the last term of the right-hand side, the update of $X_{i,k+1}$ in \eqref{re-u} implies \begin{equation}\label{2-7}
	\begin{aligned}
	\|X_{i,k_0+1} - X_{i,k_0}\|_F =&\|\mathcal{R}_{X_{i,k_0}}(-\alpha \text{grad}\ h_{i,t}(\bold{X}_{k_0})-\beta v_{i,k_0}) - X_{i,k_0}\|_F\\
	\leq& \alpha \|\text{grad} h_{i,t}(\bold{X}_{k_0})\|_F+\beta \|Y_{i,k_0}\|_F.
	\end{aligned}
	\end{equation} where the inequality follows by Lemma \ref{lem-1} and $\|v_{i,k_0}\|_F \leq \|Y_{i,k_0}\|_F$.
	
	Since $\bold{X}_{k_0+1} \in \mathcal{S}$, the term $\|\text{grad} h_{i,t}(\bold{X}_{k_0})\|_F$ is bounded by Lemma \ref{al-3}, which implies $
	\alpha\|\text{grad} h_{i,t}(\bold{X}_{k_0})\|_F \leq 2\delta_2 \alpha.$
	As $\beta \leq\frac{\alpha \delta_1}{5(10A+L_G)}$, and by using the induction hypothesis that $\|Y_{i,k}\|_F \leq 10A+L_G$, we have $
	\beta\|Y_{i,k_0}\|_F \leq \frac{1}{5}\delta_1 \alpha.$
	Substituting these results into \eqref{2-7} implies that
	\[\|X_{i,k_0+1} - X_{i,k_0}\|_F \leq 2\delta_2 \alpha+ \frac{1}{5}\delta_1 \alpha.\]
	This together with \eqref{2-3} implies $$\|Y_{i,k_0+1} - \bar F_{k_0}\|_F\leq 9A+\frac{1}{2}L_G+ 2\delta_2 L_G +\frac{L_G}{5}\delta_1 \alpha \leq 9A+L_G.$$
	
	According to Assumption \ref{assu-3}, we have $\|\bar F_k\|_F \leq A, \forall k \geq 0$, which yields $$\|Y_{i,k_0+1}\|_F \leq \|Y_{i,k_0+1} - \bar F_{k_0}\|_F+\|\bar F_{k_0}\|_F \leq 10A+L_G.$$
	
	Therefore, the proof is completed.
	
	\rightline{$\qedsymbol$}
\end{document}